\documentclass{amsart}
\usepackage[margin=4cm]{geometry}
\usepackage[T1]{fontenc}
\usepackage[utf8]{inputenc}
\usepackage{amsmath,amssymb,amsthm,mathtools}
\usepackage{enumitem}
\usepackage{microtype}
\usepackage{xcolor}
\usepackage{hyperref}
\hypersetup{
  colorlinks=true,
  linkcolor=blue,
  citecolor=blue,
  urlcolor=blue
}
\newtheorem{theorem}{Theorem}[section]
\newtheorem{proposition}[theorem]{Proposition}
\newtheorem{lemma}[theorem]{Lemma}
\newtheorem{corollary}[theorem]{Corollary}

\theoremstyle{definition}
\newtheorem{definition}[theorem]{Definition}
\newtheorem{example}[theorem]{Example}

\theoremstyle{remark}
\newtheorem{remark}[theorem]{Remark}

\newtheorem{maintheorem}{Theorem}

\newtheorem{maincorollary}[maintheorem]{Corollary}

\newcommand{\Gir}{\operatorname{Gir}}
\newcommand{\Fix}{\operatorname{Fix}}
\newcommand{\Stab}{\operatorname{Stab}}
\newcommand{\Aut}{\operatorname{Aut}}

\newcommand{\bdv}{\partial_{\!\infty}}

\newcommand{\Z}{\mathbb Z}

\newcommand{\h}{\mathfrak h}
\newcommand{\hh}{\widehat{\mathfrak h}}
\newcommand{\kk}{\mathfrak k}
\newcommand{\hk}{\widehat{\mathfrak k}}
\newcommand{\m}{\mathfrak m}
\newcommand{\hm}{\widehat{\mathfrak m}}
\newcommand{\HH}{\mathcal H}
\newcommand{\Roll}{\overline{X}^{\mathrm R}}
\newcommand{\bdreg}{\partial_{\mathrm{reg}}}
\newcommand{\bdry}{\partial_{\infty}}
\allowdisplaybreaks
\title[Boundary dynamics and infinite girth]{Boundary Dynamics, Cubical
Actions, and Infinite Girth}
\author[Amrutam]{Tattwamasi Amrutam}
\address{Institute of Mathematics of the Polish Academy of Sciences, ul. Sniadeckich 8, 00-656, Warszawa, Poland}
\email{tattwamasiamrutam@gmail.com}
\thanks{T.A. is supported by National Science Centre, Poland Sonata, grant number 2025/59/D/ST1/03117.}
\author[Banerjee]{Arka Banerjee}
\address{Department of Mathematics
School of Mathematical Sciences
Ramakrishna Mission Vivekananda Educational \& Research Institute (RKMVERI)
PO Belur Math, Dist Howrah 711202
West Bengal, India}
\email{banerjee20arka@gmail.com}
\author[Gulbrandsen]{Daniel L. Gulbrandsen}
\address{Department of Mathematics, Adams State University, Alamosa, CO 81101}
\email{dangulbrandsen@yahoo.com}
\author[Mishra]{Pratyush Mishra}
\address{HUN-REN Alfréd Rényi Institute of Mathematics, 13-15 street, Reáltanoda, 1053 Budapest, Hungary}
\email{mishrapratyushkumar@gmail.com}
\thanks{P.M. was supported by the National Research, Development and Innovation Office (NKFIH) Highlight Grant No. K153681}
\date{\today}
\subjclass[2020]{20F65, 20F67, 20F69, 37B05}
\keywords{Infinite girth, extreme boundary, CAT(0) cube complex, Roller boundary, acylindrically hyperbolic group}
\begin{document}
\begin{abstract}
We develop two methods for proving infinite girth from boundary dynamics. First, we use topologically free
extreme boundary actions to give a boundary-dynamical proof of infinite
girth for finitely generated acylindrically hyperbolic groups. The second
combines Nakamura's criterion with, respectively, flag-space and
Roller-boundary dynamics, yielding infinite-girth results for semisimple
$S$-algebraic groups and for essential non-elementary actions on
finite-dimensional, second countable, non-Euclidean CAT(0) cube complexes.
We also prove that every finitely generated large group has infinite girth,
thereby answering a question of Akhmedov and Mishra~\cite{AkhmedovMishra2026}.
\end{abstract}
\maketitle
\enlargethispage{4pt}
% \tableofcontents

\section{Introduction}\label{sec:introduction}

Let $\Gamma$ be a finitely generated group.  Following Schleimer
\cite{Schleimer2000} and Akhmedov \cite{Akhmedov2003,Akhmedov2005}, the girth of $\Gamma$ is the supremum, over all finite generating sets $S$, of
the smallest length of a nonempty cyclically reduced word in $S^{\pm1}$
representing the identity.  Thus $\Gir(\Gamma)=\infty$ means that, for every
$m\geq1$, one can choose a finite generating set for which no such word of
length at most $m$ represents the identity in $\Gamma$.

Usually, an approach to showing infinite girth involves a ping-pong argument; Nakamura's criterion \cite{Nakamura2014} provides a flexible formulation that generalizes and reformulates earlier work of
Akhmedov \cite{Akhmedov2005}. A notable exception is the
model-theoretic argument of \cite{HullOsin2016}.  Our methods are
again dynamical in spirit in that they use the dynamics of a group action, whether
on an extreme boundary, a flag space, or a Roller boundary.

Our first method can be characterized algebraically.  For $\gamma\in\Gamma$, let
$\varphi_\gamma\colon\Gamma*\langle z\rangle\to\Gamma$ be the homomorphism
which fixes $\Gamma$ pointwise and sends $z$ to $\gamma$.
A finite family of nontrivial elements of $\Gamma*\langle z\rangle$ is
required to survive under one such retraction.  Finite retraction
discrimination implies mixed-identity-freeness, and Hull--Osin proved that
every finitely generated mixed-identity-free group has infinite girth
\cite[Proposition~5.4(d)]{HullOsin2016}.  Our methods
instead give a direct construction of generating sets of arbitrarily large
girth.  Our configuration and terminology are motivated by
\cite{amrutam2025strict}, where explicit nontrivial examples of finite
retraction discrimination can be found in Section~3.

Ozawa proved that an axial sequence associated with a topologically free
extreme boundary gives asymptotically injective retractions
\cite[Lemma~6]{Ozawa2026}.  Motivated by \cite{Ozawa2026}, and using the algebraic methods introduced in \cite{amrutam2025strict}, we obtain the following.

\begin{maintheorem}\label{thm:intro-extreme}
Let $\Gamma$ be a finitely generated group admitting a topologically free
extreme boundary action on a compact Hausdorff space with more than two
points.  Then $\Gir(\Gamma)=\infty$.
\end{maintheorem}

Recall that a group is acylindrically hyperbolic if it admits a
non-elementary acylindrical action on a hyperbolic space (see~\cite{osin2016acylindrically}).  Yang proved that
every acylindrically hyperbolic group admits a compact
metrizable extreme boundary, and that the action is topologically free when
the finite radical is trivial \cite[Theorem~1.3]{Yang2026}.
Passing to the quotient by the finite radical and using Akhmedov's quotient
principle gives the following consequence.

\begin{maincorollary}\label{cor:intro-ah}
Every finitely generated acylindrically hyperbolic group has infinite girth.
\end{maincorollary}

As mentioned above, the conclusion of
Corollary~\ref{cor:intro-ah} is already known from \cite{HullOsin2016}.  Indeed, the
quotient by the finite radical is
mixed-identity-free, and a finitely generated mixed-identity-free group has
infinite girth
\cite[Corollary~5.10 and Proposition~5.4(d)]{HullOsin2016}.  Again, our proof produces explicit generating sets.

Our second method uses Nakamura's ping-pong criterion
\cite[Proposition~2]{Nakamura2014}.
In an earlier version of \cite{Ozawa2026}, Ozawa introduced a finite
ping-pong condition under the name $P_{\mathrm{PHP}}$.  Since his current
formulation is slightly different, we avoid confusion by calling the earlier
condition $P_{\mathrm{FPP}}$ (see Definition~\ref{def:fpp}).  We prove in
Theorem~\ref{thm:fpp-girth} that $P_{\mathrm{FPP}}$ implies infinite girth.

Vigdorovich extends the related flag-space construction to semisimple
$S$-algebraic groups \cite[Theorem~4.5]{vigdorovich2026}.  We use the
flag-space geometry in his proof directly and verify Nakamura's criterion,
obtaining the following.
\begin{maintheorem}\label{thm:intro-vig}
Let $S$ be finite and, for each $v\in S$, let $\mathbb K_v$ be a local field
and $\mathbf G_v$ a connected adjoint $\mathbb K_v$-simple algebraic group.
Let $G=\prod_{v\in S}\mathbf G_v(\mathbb K_v)$.  If $\Gamma<G$ is finitely
generated and its projection to every $\mathbf G_v(\mathbb K_v)$ is Zariski
dense and unbounded, then $\Gir(\Gamma)=\infty$.
\end{maintheorem}
After Theorem~\ref{thm:intro-vig} was obtained, Azer Akhmedov pointed out
that the unboundedness assumption can be removed and that Zariski density
is needed for only one coordinate projection.  His observation uses the
quotient principle rather than the Tits alternative; see
Theorem~\ref{thm:akhmedov-one-projection} below.  We thank Azer Akhmedov for
pointing out this strengthening and for suggesting its proof.

We now turn our attention to cubical dynamics of non-elementary actions. Here, non-elementary means that the action has no finite orbit in
$X\cup\bdv X$. The last three authors and Parija proved related girth
results for lattices in automorphism groups of CAT(0) cube complexes
\cite{BanerjeeGulbrandsenMishraParija2024}.  We remove the
lattice hypothesis using the random-walk machinery of
Fern\'os--L\'ecureux--Math\'eus \cite{FernosLecureuxMatheus2018} together
with a deterministic finite-coset argument.

\begin{maintheorem}\label{thm:intro-cubical}
Let $X$ be a non-Euclidean, finite-dimensional, second countable
CAT(0) cube complex.  Let $G$ be a finitely generated group acting on $X$
essentially and non-elementarily by cubical automorphisms.  Then
$\Gir(G)=\infty$.
\end{maintheorem}

Finally, we study the finite-index problem for infinite girth.  We show that
if a finite-index subgroup of a finitely generated group maps onto a
non-virtually-solvable linear group, then the ambient group has infinite
girth.  In particular, every finitely generated large group has infinite
girth, answering \cite[Question~1]{AkhmedovMishra2026}.
We finish the paper by comparing our two techniques.

\section{Girth and two general mechanisms}\label{sec:girth}

\subsection{Mixed identities and finite retraction discrimination}

Let $S$ be a finite generating set of a group $\Gamma$.  A word in the
formal alphabet $S^{\pm1}$ is reduced if no letter is immediately followed
by its formal inverse, and cyclically reduced if, in addition, the first and
last letters are not formal inverses.  We write $\Gir(\Gamma,S)$ for the
smallest length of a nonempty cyclically reduced word representing the identity,
with the convention $\Gir(\Gamma,S)=\infty$ if no such word exists.  The
girth of $\Gamma$ is $\Gir(\Gamma)=\sup_S\Gir(\Gamma,S)$, where $S$
ranges over all finite generating sets.  A few examples of groups of infinite girth are non-abelian free groups,
$\mathrm{SL}(n,\mathbf Z)$ for $n\geq2$ \cite{Akhmedov2003}, and
Thompson's group $F$, whose infinite girth was first proved in
\cite{Brin2010,AkhmedovSteinTaback2011}.
For a finitely generated group $\Gamma$, let $d(\Gamma)$ denote the minimal
cardinality of a generating set.  Following
\cite[Definition~4.7]{Akhmedov_2025}, we let
\[
\Gir_k(\Gamma)=
\sup\{\Gir(\Gamma,S):\langle S\rangle=\Gamma,\ |S|\leq k\}.
\]
\begin{definition}\label{def:discrimination}
Let $\widetilde\Gamma=\Gamma*\langle z\rangle$, where $z$ is a formal
generator of infinite order.  For $\gamma\in\Gamma$, let
$\varphi_\gamma\colon\widetilde\Gamma\to\Gamma$ be the homomorphism which
fixes $\Gamma$ pointwise and sends $z$ to $\gamma$.  We say that $\Gamma$
has \emph{finite retraction discrimination} if, for every finite set
$E\subset\widetilde\Gamma\setminus\{e\}$, there is $\gamma\in\Gamma$ such
that $\varphi_\gamma(w)\neq e$ for every $w\in E$.  Equivalently, $\Gamma$
discriminates $\Gamma*\langle z\rangle$ over the retractions
$\varphi_\gamma$. We say that the net
$(\varphi_{\gamma_\nu})_{\nu\in\mathcal D}$ is
\emph{asymptotically injective} if, for every
$w\in\widetilde\Gamma\setminus\{e\}$, one has
$\varphi_{\gamma_\nu}(w)\neq e$ eventually.
\end{definition}

We recall the relation with mixed identities.  A nontrivial group $\Gamma$
is \emph{mixed-identity-free}, or MIF, if for every nontrivial
$w\in\Gamma*\langle z\rangle$ there is $\gamma\in\Gamma$ such that
$\varphi_\gamma(w)\neq e$.  By \cite[Remark~5.1]{HullOsin2016}, this
one-variable formulation agrees with the usual definition using finitely
many variables.  Taking $E=\{w\}$ in Definition~\ref{def:discrimination}
shows that a group with finite retraction discrimination is MIF\@.  For
countable groups, the converse is also true, by
\cite[Proposition~5.3]{HullOsin2016}.

\begin{lemma}\label{lem:asymptotic-implies-finite}
If $\Gamma$ admits an asymptotically injective net of retractions, then it
has finite retraction discrimination.
\end{lemma}

\begin{proof}
Let $(\varphi_{\gamma_\nu})_{\nu\in\mathcal D}$ be an asymptotically
injective net of retractions and let
$E\subset\widetilde\Gamma\setminus\{e\}$ be finite.  For each $w\in E$
choose an index $\nu_w$ after which $\varphi_{\gamma_\nu}(w)$ is nontrivial.
Since the index set is directed and $E$ is finite, there is one index $\nu_0$
which dominates all $\nu_w$.  Then
$\varphi_{\gamma_{\nu_0}}(w)\neq e$ for every $w\in E$.
\end{proof}

The next lemma keeps track of suitable powers needed to pass from
discrimination to large girth.  Attaching two different exponents to each
generator prevents a power of a torsion generator from collapsing into a
single $\Gamma$-syllable.

\begin{lemma}
\label{lem:expansion}
Let $T=\{t_1,\ldots,t_n\}\subset\Gamma\setminus\{e\}$ be finite and let
$m\geq1$.  For $1\leq i\leq n$, set $M_i=(2i-1)(m+1)$,
$N_i=2i(m+1)$ and $X_i=z^{-M_i}t_i z^{N_i}\in\widetilde\Gamma$.
Let $W$ be a nonempty reduced word of length at most $m$ in the formal
alphabet $\{\zeta^{\pm1},\xi_1^{\pm1},\ldots,\xi_n^{\pm1}\}$, and let
$W(z)$ be the element of $\widetilde\Gamma$ obtained by substituting
$\zeta=z$ and $\xi_i=X_i$.  Then $W(z)\neq e$.  More precisely, exactly one
of the following occurs.
\begin{enumerate}[label=\textup{(\alph*)}]
\item No letter $\xi_i^{\pm1}$ occurs, and $W(z)=z^c$ for an integer $c$
with $0<|c|\leq m$.
\item Some letter $\xi_i^{\pm1}$ occurs, and $W(z)$ has reduced free-product
normal form
\begin{equation}\tag{$\ast1$}\label{eq:expanded-form}
 z^{q_0}t_{i_1}^{\epsilon_1}z^{q_1}t_{i_2}^{\epsilon_2}
 \cdots t_{i_s}^{\epsilon_s}z^{q_s},
\end{equation}
where $s\geq1$, $\epsilon_j\in\{\pm1\}$, every exponent
$q_0,\ldots,q_s$ is nonzero, and
$|q_j|\leq2n(m+1)+m$ for $0\leq j\leq s$.
\end{enumerate}
Consequently, the set $E(T,m)$ of all elements $W(z)$ arising in this way is
a finite subset of $\widetilde\Gamma\setminus\{e\}$.
\end{lemma}

\begin{proof}
If no $\xi$-letter occurs, reducedness forces $W$ to be a nonzero power of
$\zeta$, and (a) follows.  Suppose now that at least one $\xi$-letter occurs.
Collecting the maximal blocks of $\zeta$-letters, write
\begin{equation*}
 W=\zeta^{c_0}\xi_{i_1}^{\epsilon_1}\zeta^{c_1}
 \xi_{i_2}^{\epsilon_2}\cdots
 \xi_{i_s}^{\epsilon_s}\zeta^{c_s},
\end{equation*}
where $s\geq1$, $|c_j|\leq m$, and an interior exponent $c_j$ may be zero.
Since $W$ is reduced, $c_j=0$ implies
$(i_{j+1},\epsilon_{j+1})\neq(i_j,-\epsilon_j)$ for
$1\leq j\leq s-1$.

For $\epsilon\in\{\pm1\}$, write
$X_i^\epsilon=z^{-\lambda_i^\epsilon}t_i^\epsilon z^{\rho_i^\epsilon}$,
where $(\lambda_i^+,\rho_i^+)=(M_i,N_i)$ and
$(\lambda_i^-,\rho_i^-)=(N_i,M_i)$.  Substitution and collection of
adjacent powers of $z$ give the normal form as in
equation~\eqref{eq:expanded-form}, where
\begin{align*}
 q_0&=c_0-\lambda_{i_1}^{\epsilon_1},\\
 q_j&=\rho_{i_j}^{\epsilon_j}+c_j-
       \lambda_{i_{j+1}}^{\epsilon_{j+1}}
       \quad(1\leq j\leq s-1),\\
 q_s&=\rho_{i_s}^{\epsilon_s}+c_s.
\end{align*}
The estimate $|q_j|\leq2n(m+1)+m$ follows from $|c_j|\leq m$ and from
$m+1\leq\lambda_i^\epsilon,\rho_i^\epsilon\leq2n(m+1)$, the two bounds being
positive, so that
$\bigl|\rho_{i_j}^{\epsilon_j}-\lambda_{i_{j+1}}^{\epsilon_{j+1}}\bigr|
\leq(2n-1)(m+1)$.

We next show that none of the exponents $q_0,\ldots,q_s$ vanishes.  Since
$\lambda_{i_1}^{\epsilon_1}\geq m+1$ and $|c_0|\leq m$, one has
$q_0=c_0-\lambda_{i_1}^{\epsilon_1}\neq0$.  Likewise
$\rho_{i_s}^{\epsilon_s}\geq m+1$ and $|c_s|\leq m$, so
$q_s=\rho_{i_s}^{\epsilon_s}+c_s\neq0$.

For $1\leq j\leq s-1$, the $2n$ numbers $M_i$ and $N_i$ are pairwise
distinct multiples of $m+1$.  If
$\rho_{i_j}^{\epsilon_j}$ and
$\lambda_{i_{j+1}}^{\epsilon_{j+1}}$ are distinct, their difference has
absolute value at least $m+1$, and hence
$|q_j|\geq(m+1)-|c_j|\geq1$.  If they are equal, then the definitions force
$i_j=i_{j+1}$ and $\epsilon_{j+1}=-\epsilon_j$.  Since $W$ is reduced,
$c_j\neq0$, and therefore $q_j=c_j\neq0$.

Every $t_{i_j}^{\epsilon_j}$ is nontrivial and every displayed power of
$z$ is nonzero.  Thus, no two consecutive $\Gamma$-syllables can merge, and
\eqref{eq:expanded-form} is a reduced free-product normal form.  In
particular, $W(z)\neq e$.

There are only finitely many formal reduced words of length at most $m$ in
the given finite alphabet, so $E(T,m)$ is finite.
\end{proof}

\begin{theorem}
\label{thm:discrimination-girth}
Let $\Gamma$ be finitely generated.  If $\Gamma$ has finite retraction
discrimination, then
$\Gir(\Gamma)=\infty$.
\end{theorem}

\begin{proof}
It is enough to consider $m\geq2$.
The hypothesis forces $\Gamma\neq\{e\}$, since for $\Gamma=\{e\}$ no
retraction detects the nontrivial element
$z\in\Gamma*\langle z\rangle$.
Choose a finite generating set
$T=\{t_1,\ldots,t_n\}\subset\Gamma\setminus\{e\}$.  Let $E(T,m)$ be the
finite set from Lemma~\ref{lem:expansion}.  By
Definition~\ref{def:discrimination}, choose $\gamma\in\Gamma$ such that
$\varphi_\gamma(w)\neq e$ for every $w\in E(T,m)$.  For
$1\leq i\leq n$, let $x_i=\gamma^{-M_i}t_i\gamma^{N_i}$ and let
$S=\{\gamma,x_1,\ldots,x_n\}$.  Since
$t_i=\gamma^{M_i}x_i\gamma^{-N_i}$, the set $S$ generates $\Gamma$.

We first recall how the formal alphabet from
Lemma~\ref{lem:expansion} is related to the displayed generators.
For every nonempty reduced formal word $W$ of length at most $m$, its
expansion $W(z)$ belongs to $E(T,m)$, and hence
$\varphi_\gamma(W(z))\neq e$.  By construction,
$\varphi_\gamma(\zeta(z))=\gamma$ and
$\varphi_\gamma(\xi_i(z))=\gamma^{-M_i}t_i\gamma^{N_i}=x_i$.
More generally, if $W$ is a word in the formal alphabet
$\{\zeta^{\pm1},\xi_1^{\pm1},\ldots,\xi_n^{\pm1}\}$, then
$\varphi_\gamma(W(z))$ is obtained from $W$ by replacing $\zeta$ with
$\gamma$ and each $\xi_i$ with $x_i$.

Since $m\geq2$, we apply this observation to the following reduced
formal words $\{\zeta,\ \xi_i,\ \zeta\xi_i^{-1},\ \zeta\xi_i,\
\xi_i\xi_j^{-1},\ \xi_i\xi_j,\ \zeta^2,\ \xi_i^2\}$ to obtain their respective evaluations as
$\{\gamma,\ x_i,\ \gamma x_i^{-1},\ \gamma x_i,\
x_ix_j^{-1},\ x_ix_j,\ \gamma^2,\ x_i^2\}$,
where $i\neq j$. Each of the evaluations
\[
\gamma,\quad x_i,\quad \gamma x_i^{-1},\quad \gamma x_i,\quad
x_ix_j^{-1},\quad x_ix_j,\quad \gamma^2,\quad x_i^2
\]
is nontrivial.
For example, if $\gamma=x_i$, then
$\varphi_\gamma((\zeta\xi_i^{-1})(z))=\gamma x_i^{-1}=e$, contrary to
the choice of $\gamma$.  All the other possible coincidences are excluded
in the same way by the corresponding formal words in the displayed list.

Consequently, the elements
$\gamma,\gamma^{-1},x_1,x_1^{-1},\ldots,x_n,x_n^{-1}$ are pairwise
distinct.  Therefore the assignment
$\gamma^\epsilon\mapsto\zeta^\epsilon$ and
$x_i^\epsilon\mapsto\xi_i^\epsilon$, for $\epsilon\in\{\pm1\}$, is a
well-defined bijection from $S^{\pm1}$ onto
$\{\zeta^{\pm1},\xi_1^{\pm1},\ldots,\xi_n^{\pm1}\}$, and it commutes with
formal inversion.

Let $V$ be a nonempty reduced word of length at most $m$ in the alphabet
$S^{\pm1}$.  Replace each occurrence of $\gamma^{\pm1}$ by
$\zeta^{\pm1}$ and each occurrence of $x_i^{\pm1}$ by
$\xi_i^{\pm1}$, obtaining a formal word $W$.  The word $W$ is reduced. Otherwise, two consecutive translated letters would be formal inverses,
and, since the preceding bijection commutes with inversion, the
corresponding consecutive letters of $V$ would also be inverses.
Moreover, $W$ has the same length as $V$. Lemma~\ref{lem:expansion} gives $W(z)\in E(T,m)$, and by construction
$V=\varphi_\gamma(W(z))$.  The choice of $\gamma$ therefore implies that
$V\neq e$.  Thus every nonempty reduced $S$-word of length at most $m$ is
nontrivial, so $\Gir(\Gamma,S)>m$.  Since this holds for arbitrarily large
$m$, we conclude that $\Gir(\Gamma)=\infty$.
\end{proof}

\begin{remark}\label{rem:HO-redundancy}
As noted above, a group with finite retraction discrimination
is mixed-identity-free, and every finitely generated mixed-identity-free
group has infinite girth \cite[Proposition~5.4(d)]{HullOsin2016}.  We
include the direct argument because, once a retraction detecting $E(T,m)$
has been chosen, it produces a generating set with no relation
of length at most $m$.  Moreover, it only requires a single retraction to
detect that particular finite set, rather than every finite subset of
$\Gamma*\langle z\rangle$.  Nevertheless, \cite[Remark~5.5]{HullOsin2016} shows that
the number of generators may be chosen independently of $m$.
\end{remark}

\subsection{Nakamura's criterion}
The second mechanism is a dynamical criterion that concerns an action on an
arbitrary set.  We use the following result of Nakamura.

\begin{theorem}[Nakamura's criterion; \cite{Nakamura2014}]
\label{thm:nakamura}
Let $G$ act on a set $Y$, and let
$S=\{\gamma_1,\ldots,\gamma_\ell\}$ be a finite generating set.  Suppose
there are $\sigma,\tau\in G$, subsets $U_\sigma,U_\tau\subseteq Y$, and a
point $x\in Y$ such that
\begin{enumerate}
 \item $x\notin (U_\sigma\cup U_\tau)\cup
 \bigcup_{j=1}^{\ell}\bigcup_{\epsilon=\pm1}
 \gamma_j^\epsilon(U_\sigma\cup U_\tau)$,
 \label{eq:nak-1}
 \item
$\sigma^r\left(\{x\}\cup U_\tau\cup
 \bigcup_{j=1}^{\ell}\bigcup_{\epsilon=\pm1}
 \gamma_j^\epsilon U_\tau\right)
 \subseteq U_\sigma
 \quad\text{for every }r\in\Z\setminus\{0\}$,
 \label{eq:nak-2}
 \item
 $\tau^r\left(\{x\}\cup U_\sigma\cup
 \bigcup_{j=1}^{\ell}\bigcup_{\epsilon=\pm1}
 \gamma_j^\epsilon U_\sigma\right)
 \subseteq U_\tau
 \quad\text{for every }r\in\Z\setminus\{0\}$.
 \label{eq:nak-3}
\end{enumerate}
Then $G$ is noncyclic and $\Gir(G)=\infty$.
\end{theorem}
Nakamura proves the theorem by using the generating set
$\{\sigma,\tau,\widehat\gamma_1,\ldots,\widehat\gamma_\ell\}$,
with
$\widehat\gamma_j=\sigma^{p_j}\gamma_j\tau^{-p_j}$,
where the positive integers $p_j$ and their pairwise differences are larger
than a prescribed word length.  Conditions
\eqref{eq:nak-1}--\eqref{eq:nak-3} then give a ping-pong argument for every
reduced word below that length.  We shall verify these conditions directly
in Section~\ref{sec:cubical-prelim}.

We also use the following quotient permanence result, due to Akhmedov and
recorded as \cite[Proposition~5]{Nakamura2014}.

\begin{proposition}[Akhmedov's quotient principle]\label{prop:quotient}
If a finitely generated group $G$ surjects onto a noncyclic group $Q$ with
$\Gir(Q)=\infty$, then $\Gir(G)=\infty$.
\end{proposition}
\begin{remark}\label{rem:number-generators}
The preceding constructions give some information about the number of
generators needed to witness infinite girth.  Applying finite retraction
discrimination to a minimal generating set in the proof of
Theorem~\ref{thm:discrimination-girth} gives $\Gir_{d(\Gamma)+1}(\Gamma)=\infty$. Whenever Nakamura's hypotheses can be verified starting from a minimal
generating set, his construction gives
$\Gir_{d(\Gamma)+2}(\Gamma)=\infty$.
This applies to Theorems~\ref{thm:fpp-girth} and
\ref{thm:cubical-girth}, since their proofs begin with an arbitrary finite
generating set.  By contrast, the quotient principle in the form used
below is qualitative and does not by itself control the number of
generators. The optimum can equal $d(\Gamma)$. It remains interesting to
determine when the bounds $d(\Gamma)+1$ and $d(\Gamma)+2$ above can be
reduced to $d(\Gamma)$.
\end{remark}
\subsection{A finite ping-pong property}\label{sec:fpp-girth}

Ozawa introduced the following finite ping-pong condition in an
earlier version of \cite{Ozawa2026}.  To distinguish it from the current
formulation of $P_{\mathrm{PHP}}$, we denote it by $P_{\mathrm{FPP}}$.

\begin{definition}\label{def:fpp}
A group $\Gamma$ has property $P_{\mathrm{FPP}}$ if, for every finite set
$F\subseteq\Gamma\setminus\{e\}$, every $n\in\mathbb N$, and every
$\varepsilon>0$, there are a finite set $E\subseteq\Gamma$ and subsets
$B_t^+,B_t^-\subseteq\Gamma$, $t\in E$, such that:
\begin{enumerate}[label=\textup{(\arabic*)}]
\item $|E|\geq n$;
\item the family $\{B_t^+,B_t^-:t\in E\}$ is mutually disjoint;
\item $t(\Gamma\setminus B_t^-)\subseteq B_t^+$ for every $t\in E$;
\item for every $a\in F$,
\begin{equation*}
\left|\left\{(s,t)\in E^2:
 a(B_t^+\cup B_t^-)\cap(B_s^+\cup B_s^-)\neq\varnothing
\right\}\right|<\varepsilon|E|.
\end{equation*}
\end{enumerate}
\end{definition}

\begin{remark}\label{rem:fpp-zero-error}
The error term in condition~\textup{(4)} can be removed after passing to a
subfamily.  More precisely, let $F\subseteq\Gamma\setminus\{e\}$ be finite
and let $n\in\mathbb N$.  There is a finite set $E\subseteq\Gamma$, with
$|E|\geq n$, and subsets $B_t^+,B_t^-\subseteq\Gamma$, $t\in E$, satisfying
conditions~\textup{(2)} and~\textup{(3)} of Definition~\ref{def:fpp}, and
such that
\begin{equation}\tag{4'}\label{eq:fpp-zero-error}
 a(B_t^+\cup B_t^-)\cap(B_s^+\cup B_s^-)=\varnothing
\end{equation}
for every $a\in F$ and every $s,t\in E$.

Indeed, when $F=\varnothing$, apply Definition~\ref{def:fpp} with the given
$n$ and any positive error to obtain conditions~\textup{(2)} and~\textup{(3)},
while \eqref{eq:fpp-zero-error} is vacuous.  Thus assume $q:=|F|\geq1$.
Apply Definition~\ref{def:fpp} with $2n$ in place of $n$
and with $\varepsilon<1/(4q)$.  For $a\in F$, put
\begin{equation*}
 R_a=\left\{(s,t)\in E^2:
 a(B_t^+\cup B_t^-)\cap(B_s^+\cup B_s^-)\neq\varnothing\right\}.
\end{equation*}
Then $|R_a|<\varepsilon|E|$.  Let $R=\bigcup_{a\in F}R_a$, and let
$V\subseteq E$ be the set of all elements that occur as a coordinate of a
pair in $R$.  Then
\begin{equation*}
 |V|\leq2|R|\leq2\sum_{a\in F}|R_a|<2q\varepsilon|E|<\frac{|E|}{2}.
\end{equation*}
Hence $E_0:=E\setminus V$ satisfies $|E_0|>|E|/2\geq n$.  Conditions
\textup{(2)} and~\textup{(3)} remain valid after restricting to $E_0$.  If
\eqref{eq:fpp-zero-error} failed for some $a\in F$ and $s,t\in E_0$, then
$(s,t)\in R$, so $s,t\in V$, contrary to $s,t\in E_0$.  Replacing $E$ by
$E_0$ gives \eqref{eq:fpp-zero-error}.
\end{remark}

\begin{remark}\label{rem:fpp-implies-php}
It is easy to see that property $P_{\mathrm{FPP}}$ implies property
$P_{\mathrm{PHP}}$ as defined in \cite[Section~8]{Ozawa2026}. The other direction is not true in general since $P_{\mathrm{PHP}}$ is closed under product (as remarked in \cite{Ozawa2026}), whereas $P_{\mathrm{FPP}}$ is not.  
\end{remark}

\begin{theorem}\label{thm:fpp-girth}
Let $\Gamma$ be a finitely generated group with property $P_{\mathrm{FPP}}$.
Then $\Gir(\Gamma)=\infty$.
\end{theorem}

\begin{proof}
Fix a finite generating set $S=\{\gamma_1,\ldots,\gamma_\ell\}$, with the
identity omitted, and put
$F=\{\gamma_j^\epsilon:1\leq j\leq\ell,\ \epsilon=\pm1\}$.  Apply
Remark~\ref{rem:fpp-zero-error} with this $F$ and $n=3$.  Choose distinct
$\sigma,\tau,\rho\in E$, and put
$U_t=B_t^+\cup B_t^-$ for $t\in\{\sigma,\tau,\rho\}$. The set $U_\rho$ is nonempty.  Indeed, otherwise
$B_\rho^+=B_\rho^-=\varnothing$, contradicting
$\rho(\Gamma\setminus B_\rho^-)\subseteq B_\rho^+$.  Choose
$x\in U_\rho$. Condition~\textup{(2)} gives that $U_\sigma,U_\tau,U_\rho$ are pairwise
disjoint, while \eqref{eq:fpp-zero-error} gives
 $\gamma_j^\epsilon U_\sigma\cap U_\rho=
 \gamma_j^\epsilon U_\tau\cap U_\rho=\varnothing$
for every $j$ and $\epsilon=\pm1$.  Hence
$x\notin(U_\sigma\cup U_\tau)\cup
 \bigcup_{j=1}^{\ell}\bigcup_{\epsilon=\pm1}
 \gamma_j^\epsilon(U_\sigma\cup U_\tau)$,
which is the first condition of Nakamura's criterion. Again by
\eqref{eq:fpp-zero-error},
$U_\sigma\cap\gamma_j^\epsilon U_\tau=\varnothing$ for every $j$ and
$\epsilon=\pm1$.  Thus
$ \{x\}\cup U_\tau\cup
 \bigcup_{j=1}^{\ell}\bigcup_{\epsilon=\pm1}\gamma_j^\epsilon U_\tau
 \subseteq\Gamma\setminus U_\sigma$.
We now verify the required conditions.  Since
$B_\sigma^+\cap B_\sigma^-=\varnothing$, one has
$B_\sigma^+\subseteq\Gamma\setminus B_\sigma^-$.  Condition~\textup{(3)}
therefore gives $\sigma B_\sigma^+\subseteq B_\sigma^+$ and hence
$\sigma^r(\Gamma\setminus U_\sigma)\subseteq B_\sigma^+$ for every
$r\geq1$.  The same condition also implies
$\sigma^{-1}(\Gamma\setminus B_\sigma^+)\subseteq B_\sigma^-$.  Since
$B_\sigma^-\subseteq\Gamma\setminus B_\sigma^+$, it follows that
$\sigma^r(\Gamma\setminus U_\sigma)\subseteq B_\sigma^-$ for every
$r\leq-1$.  This proves the second condition of Nakamura's criterion.  The
symmetric argument gives the third condition for $\tau$.
Theorem~\ref{thm:nakamura} now gives
$\Gir(\Gamma)=\infty$.
\end{proof}
We now proceed to show that all acylindrically hyperbolic groups with trivial finite radical have property $P_{\mathrm{FPP}}$. Before that, let us briefly recall the notion of acylindrically hyperbolic groups. 

Let $(X,d)$ be a metrizable space. An action $\Gamma\curvearrowright (X,d)$ is called acylindrical if for every $\epsilon > 0$, there are $\eta, M > 0$ such that for any $x, y \in X$ with $d(x, y) \geq \eta$, the number of elements $s \in \Gamma$ satisfying $d(x, sx) \leq \epsilon$ and $d(y, sy) \leq \epsilon$ can be at most $M$. A group $\Gamma$ is called \emph{acylindrically hyperbolic} if it admits a non-elementary acylindrically hyperbolic action on a hyperbolic space.

Of course, every non-elementary hyperbolic group is acylindrically hyperbolic. Further examples of acylindrically hyperbolic groups consist of non-(virtually) cyclic groups hyperbolic relative to proper subgroups, $\text{Out}(F_n)$ for $n > 1$, many mapping class groups, and non-(virtually cyclic) groups acting properly on proper CAT($0$)-spaces and containing rank one elements, to name a few (for more details, the readers may see \cite[Section~8]{osin2016acylindrically} and the references therein).

For an acylindrically hyperbolic group $\Gamma$, let $R(\Gamma)$ denote its
finite radical, namely its maximal finite normal subgroup.  
As shown in \cite[Version~5]{Ozawa2026}, acylindrically hyperbolic groups with trivial finite radical satisfy property
$P_{\mathrm{FPP}}$.  We include a proof for completeness.  The argument
appeared in \cite[Lemma~13]{Ozawa2026} and is modeled on
\cite[Lemma~4]{10.1007/BFb0074887}.

\begin{proposition}[Ozawa]\label{prop:ah-fpp}
Let $\Gamma$ be an acylindrically hyperbolic group with trivial finite
radical.  Then $\Gamma$ has property $P_{\mathrm{FPP}}$.
\end{proposition}

\begin{proof}
Choose a hyperbolic space $X$ on which $\Gamma$ acts non-elementarily,
acylindrically, and coboundedly, and let $L=\partial X$.  By
\cite[Proposition~0.3]{abbott2019property}, the action
$\Gamma\curvearrowright L$ is minimal and topologically free, hence
strongly faithful (i.e., for every finite set
$F\subseteq\Gamma\setminus\{e\}$, there is a point of $L$ fixed by no
element of $F$).  Moreover, \cite[Proposition~1.3]{abbott2019property}
provides arbitrarily large finite families of pairwise transverse
loxodromic elements, so the action is strongly hyperbolic.

Fix a finite set $F\subseteq\Gamma\setminus\{e\}$, an integer
$n\in\mathbb N$, and $\varepsilon>0$, and put $N=\max\{n,1\}$.  By
strong faithfulness, choose $x_0\in L$ such that
$ax_0\neq x_0$ for every $a\in F$.  Since $L$ is Hausdorff and $F$ is
finite, there is an open neighborhood $V$ of $x_0$ such that $aV\cap V=\varnothing$ for every $a\in F$. By strong hyperbolicity, choose pairwise transverse loxodromic elements
$ h_0,h_1,\ldots,h_N\in\Gamma$.
By minimality, after conjugating all of them by the same element, we may
assume that $h_0^+\in V$.  Since
$h_i^\pm\neq h_0^-$ for all $1\leq i\leq N$, north--south dynamics gives
$h_0^k h_i^\pm\in V$ for every $1\leq i\leq N$ and every sufficiently
large $k$.  Fix such a $k$ and let
$ r_i=h_0^k h_i h_0^{-k}$
 for $1\leq i\leq N$.
Thus, the $r_i$ are pairwise transverse and all their attracting and
repelling fixed points belong to $V$. Choose mutually disjoint open neighborhoods $ U_i^\pm\subseteq V$ such that
$r_i^\pm\in U_i^\pm$
 for each $1\leq i\leq N$. After replacing each $r_i$ by a sufficiently large positive power,
say $t_i=r_i^{m_i}$, north--south dynamics gives us that
$ t_i(L\setminus U_i^-)\subseteq U_i^+$ for every $1\leq i\leq N$.
The elements $t_1,\ldots,t_N$ are distinct, since their fixed-point
pairs are pairwise disjoint. Let 
 $E=\{t_1,\ldots,t_N\}$.
Since
$U_i^\pm\subseteq V$, we have
that
$ a(U_i^+\cup U_i^-)\cap(U_j^+\cup U_j^-)=\varnothing$
for every $a\in F$ and every $1\leq i,j\leq N$.

Fix $y\in L$ and let
 $B_{t_i}^\pm=\{s\in\Gamma:sy\in U_i^\pm\}$ for $1\leq i\leq N$. We claim that these are the desired sets.  Indeed, $|E|=N\geq n$, and the family
$\{B_{t_i}^+,B_{t_i}^-:1\leq i\leq N\}$ is mutually disjoint.  If
$s\notin B_{t_i}^-$, then $sy\notin U_i^-$, so north--south dynamics gives
 $(t_i s)y=t_i(sy)\in U_i^+$.
Hence
 $t_i(\Gamma\setminus B_{t_i}^-)\subseteq B_{t_i}^+$.
Finally, suppose that for some $a\in F$ and some $i,j$ one had
\begin{equation*}
 a(B_{t_i}^+\cup B_{t_i}^-)
 \cap(B_{t_j}^+\cup B_{t_j}^-)\neq\varnothing.
\end{equation*}
Choose $q$ in this intersection and write $q=as$, where
$s\in B_{t_i}^+\cup B_{t_i}^-$.  Then
\begin{equation*}
 qy=a(sy)\in
 a(U_i^+\cup U_i^-)\cap(U_j^+\cup U_j^-),
\end{equation*}
contrary to our construction.  Thus the exceptional
set in condition~\textup{(4)} of Definition~\ref{def:fpp} is empty for
every $a\in F$.  Since $\varepsilon>0$ and $|E|=N\geq1$, its cardinality
is $0<\varepsilon|E|$.  This proves that $\Gamma$ has property
$P_{\mathrm{FPP}}$.
\end{proof}

\begin{remark}
Together with Theorem~\ref{thm:fpp-girth},
Proposition~\ref{prop:ah-fpp} gives another proof of
Corollary~\ref{cor:ah-girth} when the finite radical is trivial.  
\end{remark}
\subsection{Flag dynamics over local fields}\label{subsec:vig-girth}
We use the flag-space argument of Vigdorovich
\cite[Sections~2--4]{vigdorovich2026}.  Let $S$ be finite and consider an
arbitrary product
$ G=\prod_{v\in S}\mathbf G_v(\mathbb K_v)$,
where $\mathbb K_v$ is a local field and $\mathbf G_v$ is semisimple.  The
$S$-Zariski topology on $G$ is the product of the Zariski topologies on the
sets $\mathbf G_v(\mathbb K_v)$; a subgroup is $S$-Zariski dense if it is
dense for this topology \cite[Section~2]{vigdorovich2026}.  Let
$G_c=\prod_{v\in S}\mathbf G_v^\circ(\mathbb K_v)$.  For an $S$-Zariski
dense subgroup $\Gamma<G$, let $\mathcal B$ be the associated compact flag
space from \cite[Subsection~4.1]{vigdorovich2026}.  For $x\in\mathcal B$, write
$Y_x$ for the closed set of flags not transverse to $x$.

\begin{lemma}\label{lem:vig-three-position}
Assume that $\Gamma<G$ is $S$-Zariski dense and that
$\Gamma\curvearrowright\mathcal B$ is faithful.  Let $F\subset\Gamma$ be a
finite set and put $\widetilde F=F\cup\{e\}$.  Let
$\gamma_0\in\Gamma_c:=\Gamma\cap G_c$ be $\Theta_\Gamma$-proximal, with
transverse attracting and repelling points $(x^+,x^-)$.  Then there are
$s_1,s_2,s_3\in\Gamma_c$ such that:
\begin{enumerate}[label=\textup{(\arabic*)}]
\item the points $a s_i x^\delta$, where $a\in\widetilde F$,
$1\leq i\leq3$ and $\delta\in\{+,-\}$, are pairwise distinct;
\item for $i\neq j$, $a\in\widetilde F$ and
$\delta,\eta\in\{+,-\}$,
 $a s_j x^\eta\notin Y_{s_i x^\delta}$.

\end{enumerate}
\end{lemma}
\begin{proof}
Apply \cite[Proposition~4.4]{vigdorovich2026} with the finite set
$\widetilde F$, with $E=\{x^+,x^-\}$, and with $n=3$.  Its kernel
hypothesis holds since $\ker(G\curvearrowright\mathcal B)
 \cap\widetilde F^{-1}\widetilde F
 \subseteq
 \ker(\Gamma\curvearrowright\mathcal B)
 =\{e\}$.
Consequently, there is a nonempty $S$-Zariski-open subset
$\Omega_0\subseteq G_c^3$ such that, for every
$(g_1,g_2,g_3)\in\Omega_0$, the points
$\{ ag_i x^\delta,:~
 a\in\widetilde F,~ 1\leq i\leq3,~
 \delta\in\{+,-\}
\}$
are pairwise distinct.  The same-index collision case in
\cite[Proposition~4.4]{vigdorovich2026} uses the self-normalization of parabolic point stabilizers
through \cite[Lemma~2.3]{vigdorovich2026}. We now impose the additional transversality conditions.  Fix
$a\in\widetilde F$, distinct indices $i,j$, and
$\delta,\eta\in\{+,-\}$.  Define the set
\[
 Z_{a,i,j}^{\delta,\eta}
 :=
 \left\{
 (g_1,g_2,g_3)\in G_c^3:
 ag_jx^\eta\in Y_{g_ix^\delta}
 \right\}.
\]
The set
 $\mathcal I
 =
 \{(u,v)\in\mathcal B^2:v\in Y_u\}$
is $S$-Zariski closed by the Bruhat description of the
nontransversality loci
\cite[Section~3]{vigdorovich2026}.  Therefore, it follows that
$Z_{a,i,j}^{\delta,\eta}$ is $S$-Zariski closed. It is proper.  Indeed, set $g_i=e$ and fix the unused coordinate.  Since $G_c$
acts transitively on $\mathcal B$, the map $g_j\longmapsto ag_jx^\eta
$
is onto $\mathcal B$.  Since
$Y_{x^\delta}\subsetneq\mathcal B$, we can choose $g_j$ such that
$ag_jx^\eta\notin Y_{x^\delta}$.
Thus $Z_{a,i,j}^{\delta,\eta}\neq G_c^3$. There are only finitely many choices of
$(a,i,j,\delta,\eta)$.  As such, 
\[
 \Omega
 :=
 \Omega_0
 \cap
 \bigcap_{\substack{
       a\in\widetilde F,\ i\neq j\\
       \delta,\eta\in\{+,-\}}}
 \left(G_c^3\setminus Z_{a,i,j}^{\delta,\eta}\right)
\]
being a finite intersection of nonempty $S$-Zariski-open subsets of the
irreducible space $G_c^3$ is nonempty. Since
$\Gamma_c^3$ is Zariski dense in $G_c^3$,   we may therefore choose
$(s_1,s_2,s_3)\in\Gamma_c^3\cap\Omega$.
By construction, this triple satisfies both conclusions.
\end{proof}
\begin{theorem}\label{thm:vig-simple-girth}
Let $S$ be finite and, for each $v\in S$, let $\mathbb K_v$ be a local
field and $\mathbf G_v$ a connected adjoint $\mathbb K_v$-simple algebraic
group.  Let $G=\prod_{v\in S}\mathbf G_v(\mathbb K_v)$.  If $\Gamma<G$ is
finitely generated and its projection to every $\mathbf G_v(\mathbb K_v)$
is Zariski dense and unbounded, then $\Gir(\Gamma)=\infty$.
\end{theorem}

\begin{proof}
As in the proof of \cite[Theorem~1.3]{vigdorovich2026}, the hypotheses imply
that $\Gamma$ is $S$-Zariski dense and that the associated flag action is
faithful.
Fix a finite symmetric generating set $F\subset\Gamma\setminus\{e\}$ and
put $\widetilde F=F\cup\{e\}$.  By \cite[Lemma~4.2]{vigdorovich2026},
choose $\gamma_0\in\Gamma_c$ such that both $\gamma_0$ and
$\gamma_0^{-1}$ are $\Theta_\Gamma$-proximal, with transverse attracting
and repelling points $(x^+,x^-)$.  Using
Lemma~\ref{lem:vig-three-position}, we obtain $s_1,s_2,s_3\in\Gamma_c$ that
satisfy its conditions (1) and (2). Let $x_i^\pm=s_i x^\pm$ for each $i=1,2,3$.
Since $\mathcal B$ is compact
Hausdorff, hence normal, we can choose open neighborhoods $U_i^\pm$ of
$x_i^\pm$ and open neighborhoods
$V_i^\pm$ of $Y_{x_i^\pm}$ so that the following hold.
\begin{enumerate}[label=\textup{(\roman*)}]
\item the sets $aU_i^\pm$, for $a\in\widetilde F$ and $1\leq i\leq3$,
are pairwise disjoint;
\item $aU_j^\eta\cap V_i^\delta=\varnothing$ whenever $i\neq j$,
$a\in\widetilde F$ and $\delta,\eta\in\{+,-\}$;
\item $U_i^+\cap V_i^-=U_i^-\cap V_i^+=\varnothing$ for $1\leq i\leq3$.
\end{enumerate}
The required separations follow from Lemma~\ref{lem:vig-three-position}
and the transversality of $(x_i^+,x_i^-)$.

By the uniform convergence in \cite[Fact~4.1]{vigdorovich2026}, after
choosing $m_i$ sufficiently large and putting
$\gamma_i=s_i\gamma_0^{m_i}s_i^{-1}$, we obtain that
$\gamma_i(\mathcal B\setminus V_i^-)\subseteq U_i^+$, and
$\gamma_i^{-1}(\mathcal B\setminus V_i^+)\subseteq U_i^-$.
Let $U_i=U_i^+\cup U_i^-$ and $V_i=V_i^+\cup V_i^-$.  By condition
\textup{(iii)}, $U_i^+\subseteq\mathcal B\setminus V_i^-$ and
$U_i^-\subseteq\mathcal B\setminus V_i^+$. Therefore, it follows that
$\gamma_i^r(\mathcal B\setminus V_i)\subseteq U_i$ for all
$r\in\mathbb Z\setminus\{0\}$. We now apply Nakamura's criterion to the
action on the space $\mathcal B$.  Let
$\sigma=\gamma_1$, $\tau=\gamma_2$ and choose $x=x_3^+\in U_3^+$.  By
\textup{(i)}, it follows that
 $x\notin(U_1\cup U_2)\cup\bigcup_{a\in F}a(U_1\cup U_2)$.
By \textup{(ii)}, we obtain that
$ \{x\}\cup U_2\cup\bigcup_{a\in F}aU_2
 \subseteq\mathcal B\setminus V_1$,
and hence, for all $r\in\mathbb{Z}\setminus\{0\}$,
\[\sigma^r\left(\{x\}\cup U_2\cup\bigcup_{a\in F}aU_2\right)=\gamma_1^r\left(\{x\}\cup U_2\cup\bigcup_{a\in F}aU_2\right)\subseteq\gamma_1^r(\mathcal B\setminus V_1)\subset U_1.\]
This ensures that Nakamura's second condition is satisfied.
A symmetric argument shows that for all $r\in\mathbb{Z}\setminus\{0\}$,
\[\tau^r\left(\{x\}\cup U_1\cup\bigcup_{a\in F}aU_1\right)=\gamma_2^r\left(\{x\}\cup U_1\cup\bigcup_{a\in F}aU_1\right)\subseteq\gamma_2^r(\mathcal B\setminus V_2)\subset U_2.\]
Thus, Nakamura's third and final condition is satisfied.  It follows from
Theorem~\ref{thm:nakamura} that $\Gir(\Gamma)=\infty$.
\end{proof}
The preceding theorem admits the following strengthening, pointed out to us
by Azer Akhmedov.  In particular, one may remove the unboundedness
assumption and require Zariski density for only one coordinate projection. We thank him for this observation and its proof.
\begin{theorem}
\label{thm:akhmedov-one-projection}
Let $S$ be finite and, for each $v\in S$, let $\mathbb K_v$ be a local
field and $\mathbf G_v$ a connected adjoint $\mathbb K_v$-simple algebraic
group.  Let
\[
G=\prod_{v\in S}\mathbf G_v(\mathbb K_v).
\]
If $\Gamma<G$ is finitely generated and its projection to
$\mathbf G_{v_0}(\mathbb K_{v_0})$ is Zariski dense for some $v_0\in S$,
then $\Gir(\Gamma)=\infty$.
\end{theorem}

\begin{proof}
Let
$Q=\pi_{v_0}(\Gamma)
  \leq \mathbf G_{v_0}(\mathbb K_{v_0})$.
Then $Q$ is a finitely generated linear group.  We claim that $Q$ is not
virtually solvable.  Suppose otherwise, and let $A\leq Q$ be a solvable
subgroup of finite index.  Write
$Q=q_1A\cup\cdots\cup q_mA$
and let $\mathbf H$ be the Zariski closure of $A$ in
$\mathbf G_{v_0}$.  The algebraic group $\mathbf H$ is solvable.  Since
$Q$ is Zariski dense, we obtain
\[
\mathbf G_{v_0}
 =q_1\mathbf H\cup\cdots\cup q_m\mathbf H.
\]
A connected algebraic group is irreducible and cannot be a finite union of
proper Zariski-closed subsets.  Consequently
$\mathbf H=\mathbf G_{v_0}$, contradicting the fact that
$\mathbf G_{v_0}$ is a non-solvable simple algebraic group.  Thus $Q$ is
not virtually solvable. It follows from
\cite[Theorem~4.4]{Akhmedov2005} gives $\Gir(Q)=\infty$.  In particular,
$Q$ is noncyclic.  Since
$\pi_{v_0}\colon\Gamma\twoheadrightarrow Q$ is surjective,
Proposition~\ref{prop:quotient} gives $\Gir(\Gamma)=\infty$.
\end{proof}

\begin{remark}\label{rem:theorem-c-nonlinear}
The groups covered by Theorem~\ref{thm:vig-simple-girth}, and even by its
strengthening above, need not themselves be linear over any field.  
\end{remark}

\section{Extreme boundaries and asymptotically injective retractions}
\label{sec:extreme}

Throughout this section, $\Gamma$ is a countable group acting by
homeomorphisms on a compact Hausdorff space $M$.
The results below were proved by Ozawa \cite{Ozawa2026}.  We include the details for completeness.

\begin{definition}\label{def:extreme-action}
The action $\Gamma\curvearrowright M$ is
\begin{enumerate}[label=\textup{(\arabic*)}]
\item \emph{minimal} if every orbit is dense;
\item \emph{extremely proximal} if, for every two nonempty open subsets
$U,V\subseteq M$, there is $g\in\Gamma$ such that
$g(M\setminus U)\subseteq V$;
\item an \emph{extreme boundary action} if it is minimal and extremely
proximal;
\item \emph{topologically free} if $\Fix(g)$ has empty interior for every
$g\in\Gamma\setminus\{e\}$.
\end{enumerate}
\end{definition}

For a countable group acting on a compact Hausdorff space, topological
freeness is equivalent to saying that the points with trivial stabilizer are dense.
\begin{lemma}\label{lem:perfect-extreme}
If $\Gamma\curvearrowright M$ is extremely proximal and $|M|>2$, then $M$
has no isolated points.
\end{lemma}

\begin{proof}
Suppose $p\in M$ were isolated.  Then $V=\{p\}$ is open.  Choose distinct
points $a,b\in M\setminus\{p\}$ and put $U=M\setminus\{a,b\}$.  The set $U$
is nonempty and open.  Extreme proximality gives $g\in\Gamma$ with
$g\{a,b\}=g(M\setminus U)\subseteq\{p\}$, contradicting injectivity of
the homeomorphism $g$.
\end{proof}

\begin{definition}[Axial net; cf. \cite{Ozawa2026}]
\label{def:axial-net}
A net $(\gamma_\nu)_{\nu\in\mathcal D}$ in $\Gamma$ is \emph{axial} if
there are distinct points $z_+,z_-\in M$ such that the following conditions
hold.
\begin{enumerate}[label=\textup{(\arabic*)}]
\item For every pair of neighborhoods $U_+$ of $z_+$ and $U_-$ of $z_-$,
eventually
\begin{equation}\tag{$\ast3$}\label{eq:axial-dynamics}
 \gamma_\nu(M\setminus U_-)\subseteq U_+,
 \qquad
 \gamma_\nu^{-1}(M\setminus U_+)\subseteq U_-.
\end{equation}
\item The action of $\Gamma$ on the pair $\{z_+,z_-\}$ is free: if
$g\{z_+,z_-\}\cap\{z_+,z_-\}\neq\varnothing$, then $g=e$.
\end{enumerate}
The points $z_+$ and $z_-$ are called the poles of the net.  For
$\epsilon\in\{\pm1\}$, we write $z_\epsilon$ for the corresponding pole and
$\bar\epsilon=-\epsilon$.
\end{definition}

\begin{proposition}[Ozawa]\label{prop:axial-net-exists}
Suppose $\Gamma\curvearrowright M$ is topologically free and extremely
proximal, and $|M|>2$.  Then $\Gamma$ admits an axial net.  If $M$ is first
countable, the net may be chosen to be a sequence.
\end{proposition}

\begin{proof}
By Lemma~\ref{lem:perfect-extreme}, every singleton is closed with empty
interior.  Since $\Gamma$ is countable, every orbit is meager.  The set
$M_{\mathrm{free}}$ defined above is a dense $G_\delta$.

Choose $z_+\in M_{\mathrm{free}}$.  The set
$M_{\mathrm{free}}\setminus\Gamma z_+$ is nonempty by the Baire theorem, so
choose $z_-\in M_{\mathrm{free}}\setminus\Gamma z_+$.  We verify the
freeness condition for the pair.  If $gz_+=z_+$ or $gz_-=z_-$, then $g=e$
because both points have a trivial stabilizer.  The equalities $gz_+=z_-$ and
$gz_-=z_+$ are impossible because they would put $z_-$ in the orbit of
$z_+$.

Let $\mathcal D$ be the set of pairs $\nu=(U_+,U_-)$, where $U_\pm$ is an
open neighborhood of $z_\pm$.  Direct $\mathcal D$ by reverse inclusion in
both coordinates, i.e., $(V_+,V_-)\geq(U_+,U_-)$ when
$V_+\subseteq U_+$ and $V_-\subseteq U_-$.  For each $\nu=(U_+,U_-)$,
extreme proximality supplies $\gamma_\nu\in\Gamma$ such that
$\gamma_\nu(M\setminus U_-)\subseteq U_+$.  Since $\gamma_\nu$ is a
bijection, taking complements gives
$\gamma_\nu^{-1}(M\setminus U_+)\subseteq U_-$.  If $V_\pm$ are fixed
neighborhoods of $z_\pm$, then \eqref{eq:axial-dynamics} holds for every
$\nu=(U_+,U_-)\geq(V_+,V_-)$.  Thus $(\gamma_\nu)$ is axial.  If $M$ is
first countable, choose decreasing countable neighborhood bases at $z_+$ and
$z_-$ and run the same construction along them.
\end{proof}
Let
$\widetilde\Gamma=\Gamma*\langle z\rangle$ and let
$\varphi_\nu=\varphi_{\gamma_\nu}$.
Every element of $\widetilde\Gamma\setminus\Gamma$ can be written in the
form
\begin{equation}\tag{$\ast4$}\label{eq:ozawa-spelling}
 w=a_0z^{\epsilon_1}a_1z^{\epsilon_2}\cdots
 a_{r-1}z^{\epsilon_r}a_r,
\end{equation}
where $r\geq1$, $a_i\in\Gamma$, and
$\epsilon_i\in\{\pm1\}$.  We require the no-cancellation condition
\begin{equation}\tag{$\ast5$}\label{eq:no-cancellation}
 a_i=e\quad\Longrightarrow\quad
 \epsilon_i=\epsilon_{i+1}\qquad(1\leq i\leq r-1).
\end{equation}
Such a spelling is obtained by writing the usual free-product normal form
and splitting every nonzero power of $z$ into letters $z^{\pm1}$.
Conversely, \eqref{eq:no-cancellation} guarantees that
\eqref{eq:ozawa-spelling} represents an element outside $\Gamma$.

\begin{lemma}[Ozawa]
\label{lem:ozawa-convergence}
Let $(\gamma_\nu)$ be an axial net with poles $z_\pm$, and let $w$ be
spelled as in \eqref{eq:ozawa-spelling}--\eqref{eq:no-cancellation}.  If
$x\neq a_r^{-1}z_{\bar\epsilon_r}$, then
$\varphi_\nu(w)x\to a_0z_{\epsilon_1}$.
\end{lemma}

\begin{proof}
We argue by induction on $r$.

Suppose first that $r=1$, so $w=a_0z^{\epsilon_1}a_1$.  The hypothesis says
that $a_1x\neq z_{\bar\epsilon_1}$.  Let $U$ be a neighborhood of
$a_0z_{\epsilon_1}$.  Then $a_0^{-1}U$ is a neighborhood of
$z_{\epsilon_1}$.  Choose a neighborhood $V$ of
$z_{\bar\epsilon_1}$ which does not contain $a_1x$.  It follows from the Axial condition that
$\gamma_\nu^{\epsilon_1}(M\setminus V)\subseteq a_0^{-1}U$ eventually.
Since $a_1x\notin V$, it follows that
$a_0\gamma_\nu^{\epsilon_1}a_1x\in U$ eventually.  This proves the base
case.

Assume $r\geq2$ and put
$v=a_1z^{\epsilon_2}a_2\cdots z^{\epsilon_r}a_r$.  The tail has length
$r-1$ and satisfies the same no-cancellation condition.  By the induction
hypothesis, $\varphi_\nu(v)x\to y:=a_1z_{\epsilon_2}$.

We claim that $y\neq z_{\bar\epsilon_1}$.  If $a_1\neq e$, an equality
$a_1z_{\epsilon_2}=z_{\bar\epsilon_1}$ would contradict the freeness of the
action on $\{z_+,z_-\}$.  If $a_1=e$, condition
\eqref{eq:no-cancellation} gives $\epsilon_1=\epsilon_2$, and hence
$y=z_{\epsilon_1}\neq z_{\bar\epsilon_1}$.

Let $U$ be a neighborhood of $a_0z_{\epsilon_1}$.  Since $M$ is Hausdorff,
choose disjoint neighborhoods $V$ of $z_{\bar\epsilon_1}$ and $V'$ of
$y$.  The induction hypothesis gives
$\varphi_\nu(v)x\in V'\subseteq M\setminus V$ eventually, while axiality
gives $\gamma_\nu^{\epsilon_1}(M\setminus V)\subseteq a_0^{-1}U$
eventually.  Since the index set is directed, both statements hold
simultaneously.  Therefore
$\varphi_\nu(w)x=a_0\gamma_\nu^{\epsilon_1}(\varphi_\nu(v)x)\in U$
eventually, as required.
\end{proof}

\begin{theorem}\label{thm:axial-asymptotic}
Let $\Gamma\curvearrowright M$ be a countable group action on a compact
Hausdorff space with $|M|>2$, and suppose that the action admits an axial
net.  Then the associated retractions
$\varphi_\nu\colon\Gamma*\langle z\rangle\to\Gamma$ are asymptotically
injective.
\end{theorem}

\begin{proof}
Let $w\in\widetilde\Gamma\setminus\{e\}$.  If $w\in\Gamma$, then
$\varphi_\nu(w)=w\neq e$ for every $\nu$.  Suppose $w\notin\Gamma$ and
choose a spelling \eqref{eq:ozawa-spelling}.  Since $|M|>2$, choose
$x\in M\setminus\{a_r^{-1}z_{\bar\epsilon_r},a_0z_{\epsilon_1}\}$.
By Lemma~\ref{lem:ozawa-convergence},
$\varphi_\nu(w)x\to a_0z_{\epsilon_1}\neq x$.  Choose disjoint
neighborhoods of $x$ and $a_0z_{\epsilon_1}$.  Eventually
$\varphi_\nu(w)x$ lies in the second neighborhood, and therefore differs
from $x$.  Thus $\varphi_\nu(w)\neq e$ eventually.
\end{proof}

\begin{theorem}\label{thm:extreme-girth}
Let $\Gamma$ be a finitely generated group admitting a topologically free
extreme boundary action on a compact Hausdorff space with more than two
points.  Then $\Gir(\Gamma)=\infty$.
\end{theorem}

\begin{proof}
A finitely generated group is countable.  By
Proposition~\ref{prop:axial-net-exists}, the action supplies an axial net.
Theorem~\ref{thm:axial-asymptotic} and
Lemma~\ref{lem:asymptotic-implies-finite} gives finite retraction
discrimination.  Theorem~\ref{thm:discrimination-girth} now gives infinite
girth.
\end{proof}
\begin{remark}\label{rem:law-obstruction} 
Suppose a countable group $\Gamma$ acts topologically freely and extremely
proximally on a compact Hausdorff space $M$ with $|M|>2$. It is easy to see that such an
action is strongly proximal in the sense that for every $\mu\in\text{Prob}(M)$, there is
$x\in M$ such that
$\delta_x\in\overline{\Gamma\mu}^{\,w^*}$. The action is therefore a topologically free boundary action, so $\Gamma$
is $C^*$-simple (see ~\cite{MR3652252},~\cite{MR3735864}). The converse fails in the sense that not every $C^*$-simple group admits an extremely proximal boundary action which is topologically free.  For every $m\geq2$ and every sufficiently
large odd $n$, the reduced $C^*$-algebra of the free Burnside group $B(m,n)$
is simple (see~\cite[Theorem~1.2]{MR3267529}).  Thus $B(m,n)$ is $C^*$-simple and
acts topologically freely, minimally, and strongly proximally on its
Every finite generating set of
$B(m,n)$ contains an element $s\neq e$, and $s^n$ is a nonempty cyclically
reduced word of length $n$ representing the identity.  Hence
$\Gir(B(m,n))\leq n$.
\end{remark}
It is well-known that $\overline\Gamma:=\Gamma/R(\Gamma)$ is acylindrically hyperbolic
and has trivial finite radical (see~\cite[Lemma~2.9]{HullOsin2016}). 
\begin{corollary}\label{cor:ah-girth}
Every finitely generated acylindrically hyperbolic group has infinite girth.
\end{corollary}

\begin{proof}
Let $\Gamma$ be finitely generated and acylindrically hyperbolic, and let
$\overline\Gamma=\Gamma/R(\Gamma)$. The quotient is finitely generated, acylindrically hyperbolic, and has no
nontrivial finite normal subgroup.  In particular, $\overline\Gamma$ is not
virtually cyclic and hence is infinite. Using~\cite[Theorem 1.3]{Yang2026}, we see that it admits
a topologically free extreme boundary action on a nonempty compact
metrizable space $M$.  The space $M$ has more than two points.  Indeed, if
$|M|\leq2$, the
kernel of the action on the finite set $M$ would be a finite-index subgroup
fixing every point.  Topological freeness would force this kernel to be
trivial, making $\overline\Gamma$ finite, a contradiction.  Hence
Theorem~\ref{thm:extreme-girth} gives $\Gir(\overline\Gamma)=\infty$, and
$\overline\Gamma$ is noncyclic.  The quotient map
$\Gamma\twoheadrightarrow\overline\Gamma$ and
Proposition~\ref{prop:quotient} therefore imply
$\Gir(\Gamma)=\infty$.
\end{proof}

\section{Roller boundaries, finite general position and infinite girth}\label{sec:cubical-prelim}
The preceding section treats infinite girth through extreme-boundary
dynamics.  The cubical setting is genuinely different.  A group acting
essentially and non-elementarily on a CAT(0) cube complex need not be
acylindrically hyperbolic (case in point being that of $F_2\times F_2$ acting properly,
essentially, and non-elementarily on the product of its two Cayley trees).   The relevant dynamics are also different.  Instead of extreme proximality,
we use the combinatorics of halfspaces and the dynamics of regular points in
the Roller boundary.
\subsection{CAT(0) cube complexes and halfspaces}\label{subsec:cube-basics}

We recall the cubical material used in the proof of
Theorem~\ref{thm:intro-cubical}, and refer the readers to
\cite{Sageev1995,Roller2016,CapraceSageev2011,ChatterjiFernosIozzi2016,FernosLecureuxMatheus2018} for more details.
A cube complex is obtained by gluing Euclidean unit cubes along faces by
isometries. It is called a CAT(0) cube complex when its induced Euclidean length
metric is CAT(0). We write $X^0$ for its vertex set and use the combinatorial
metric on $X^0$, namely the path metric in the one-skeleton.

A midcube of a Euclidean cube is obtained by setting one coordinate equal to
$1/2$. A hyperplane $\hh$ is a connected union of midcubes that meets each
cube in either one midcube or the empty set and such that the intersection of two
midcubes in $\hh$ is either empty or a common face of both. Every hyperplane
separates a CAT(0) cube
complex into two components \cite{Sageev1995}. The corresponding halfspaces
will always be regarded as subsets of $X^0$. Thus, if $\h$ and $\h^*$ are the
two halfspaces bounded by $\hh$, then $\h\cap\h^*=\varnothing$ and
$\h\cup\h^*=X^0$. Both halfspaces are nonempty; in particular, whenever we
choose a vertex in a halfspace below, such a vertex exists. Two hyperplanes
are transverse if they meet. Equivalently, if $\h,\h^*$ and
$\kk,\kk^*$ are their halfspaces, all four intersections
$\h^{(*)}\cap\kk^{(*)}$ are nonempty. Two hyperplanes are strongly separated
if they are disjoint, and no hyperplane is transverse to both. Halfspaces are
strongly separated when their bounding hyperplanes are strongly separated.

Throughout Section~\ref{sec:cubical-prelim}, $X$
is a finite-dimensional, second countable CAT(0) cube complex, and all actions
are by cubical automorphisms. Second countability is equivalent to
countability of $\HH(X)$; it also implies that $X^0$ is countable.  This is
the standing convention of
\cite{FernosLecureuxMatheus2018}, whose results we quote below.  By
\cite[Remark~3.1]{FernosLecureuxMatheus2018}, it is not a real restriction since a
countable group acting on a finite-dimensional CAT(0) cube complex preserves a
second countable subcomplex.  We assume neither local finiteness nor
properness.

We shall repeatedly use the following elementary observation.

\begin{lemma}
\label{lem:intermediate}
Let $\h_1\subseteq\h_2\subseteq\h_3$ be halfspaces. If a hyperplane $\hm$ is
transverse to both $\widehat{\h_1}$ and $\widehat{\h_3}$, then it is
transverse to $\widehat{\h_2}$.
\end{lemma}

\begin{proof}
Let $\m,\m^*$ be the two halfspaces bounded by $\hm$.  Since $\hm$ is
transverse to $\widehat{\h_1}$, all four quadrants determined by $\m$ and
$\h_1$ are nonempty.  In particular, both $\m\cap\h_1$ and $\m^*\cap\h_1$ are
nonempty.  Since $\h_1\subseteq\h_2$, it follows that both $\m\cap\h_2$ and
$\m^*\cap\h_2$ are nonempty.  Similarly, transversality of $\hm$ and
$\widehat{\h_3}$ gives nonempty sets $\m\cap\h_3^*$ and $\m^*\cap\h_3^*$.
Since $\h_3^*\subseteq\h_2^*$, both $\m\cap\h_2^*$ and $\m^*\cap\h_2^*$ are
nonempty.  Thus all four quadrants determined by $\m$ and $\h_2$ are
nonempty, which is equivalent to transversality of $\hm$ and
$\widehat{\h_2}$.
\end{proof}

\begin{corollary}\label{cor:pairwise-strong}
If $\h_1\supsetneq\h_2\supsetneq\cdots$ is a descending chain and every consecutive pair of hyperplanes is strongly separated, then every two
distinct hyperplanes in the chain are strongly separated.
\end{corollary}

\begin{proof}
Suppose $i<j$ and a hyperplane $\hm$ is transverse to both $\hh_i$ and
$\hh_j$. If $j=i+1$, this contradicts the hypothesis. If $j>i+1$,
Lemma~\ref{lem:intermediate}, applied to
$\h_j\subseteq\h_{i+1}\subseteq\h_i$, shows that $\hm$ is transverse to
$\hh_{i+1}$. It is then transverse to both members of the consecutive pair
$\hh_i,\hh_{i+1}$, again a contradiction. The hyperplanes $\hh_i$ and $\hh_j$ are disjoint because their halfspaces are properly nested. Hence, they are
strongly separated.
\end{proof}

By \cite[Proposition~2.6]{CapraceSageev2011}, an essential
finite-dimensional CAT(0) cube complex admits a canonical decomposition
$X=X_1\times\cdots\times X_p$ into irreducible factors, unique up to
permutation. Every cubical automorphism permutes the factors. If a group $G$
acts on $X$ by cubical automorphisms, we denote by $G^0$ the kernel of the
induced permutation action on the factors. Thus, $G^0$ is a finite-index
normal subgroup. Moreover, by the canonical product decomposition, every
element of $G^0$ acts factorwise, hence diagonally, on
$X_1\times\cdots\times X_p$; see \cite[Proposition~2.6]{CapraceSageev2011}.
The action also induces actions on the disjoint unions
$\bigsqcup_k\bdreg X_k$ and $\bigsqcup_k\bdry X_k$; an automorphism carrying
$X_k$ to $X_r$ carries the corresponding boundary of $X_k$ to that of $X_r$.

The visual boundary $\bdry X$ is the set of asymptotic classes of CAT(0)
geodesic rays. The action $G\curvearrowright X$ is \emph{essential} if, for one
and hence every vertex $v$, both $\sup_{g\in G}d(gv,\h^*)$ and
$\sup_{g\in G}d(gv,\h)$ are infinite for every halfspace $\h$. Thus, every
orbit enters arbitrarily
deeply into both sides of every hyperplane. The action is
\emph{non-elementary} if it has no finite orbit in $X\cup\bdry X$. The
complex $X$ is \emph{essential} if its full cubical automorphism group acts
essentially. These are the conventions of
\cite{CapraceSageev2011,FernosLecureuxMatheus2018}. We call $X$
\emph{non-Euclidean} if no irreducible factor is quasi-isometric to a line.
Note that if some $G\leq\Aut(X)$ acts essentially, then $\Aut(X)$ does too,
so $X$ is essential; we therefore omit that adjective from the hypotheses of
the statements below in which a group is assumed to act essentially.

\subsection{The Roller compactification}\label{subsec:roller}

Let $\HH(X)$ denote the set of halfspaces of $X$. An ultrafilter is a subset
$\alpha\subseteq\HH(X)$ satisfying the following two conditions.
\begin{enumerate}[label=\textup{(U\arabic*)}]
\item For every hyperplane, exactly one of its two halfspaces belongs to
$\alpha$.
\item If $\h\in\alpha$ and $\h\subseteq\kk$, then $\kk\in\alpha$.
\end{enumerate}
The set of all ultrafilters, with the topology inherited from
$\{0,1\}^{\HH(X)}$, is the Roller compactification $\Roll$. It is compact and
Hausdorff. A vertex $v$ determines the principal ultrafilter
$\alpha_v=\{\h:v\in\h\}$. Principal ultrafilters are exactly the ultrafilters
satisfying the descending chain condition.  The resulting map
$e\colon X^0\to\Roll$, $v\mapsto\alpha_v$, is injective, and its image is
dense because finitely many pairwise intersecting halfspaces have a common
vertex.  We identify $X^0$ with the set of principal ultrafilters. The Roller
boundary is $\partial X=\Roll\setminus X^0$.
The vertex map $e\colon X^0\to\{0,1\}^{\HH(X)}$ is the construction in
\cite[Section~3.1, Definition~3.2]{FernosLecureuxMatheus2018}.  When $X$ is
not locally finite, the subspace topology on $e(X^0)$ need not agree with the
discrete topology on the vertex set.  Thus, we regard $\Roll$ as the closure
of the image $e(X^0)$.  Compare
\cite[Remark~6.19 and Proposition~6.20]{FernosLecureuxMatheus2018}, which describes the corresponding nonproper-space caveat through the
horocompactification.  Every statement below concerns the compact $G$-space
$\Roll$ and the points $e(v)$, and we suppress $e$ from the notation.

For a halfspace $\h$, put
$\mathcal U_\h=\{\alpha\in\Roll:\h\in\alpha\}$.  The set
$\mathcal U_\h$ is clopen, and the sets obtained by taking finite
intersections of such sets form a basis. If $\h\subseteq\kk$, then
$\mathcal U_\h\subseteq\mathcal U_\kk$. Two halfspaces belonging to one
ultrafilter always meet. Indeed, if $\h,\kk\in\alpha$ and
$\h\cap\kk=\varnothing$, then $\h\subseteq\kk^*$, so upward consistency gives
$\kk^*\in\alpha$, contrary to the choice condition.

The following nesting lemma is a slight reformulation of Lemma~5.11 in
\cite{FernosLecureuxMatheus2018}.

\begin{lemma}\label{lem:nesting}
Let $\h_1\supsetneq\h_2\supsetneq\cdots$ be an infinite descending chain of
halfspaces, and let $\kk$ be a halfspace such that $\kk\cap\h_i\neq\varnothing$
for every $i$. Then exactly one of the following alternatives occurs after
discarding finitely many terms.
\begin{enumerate}[label=\textup{(\arabic*)}]
\item The hyperplane $\hk$ is transverse to every $\hh_i$.
\item One has $\h_i\subseteq\kk$ for every $i$.
\end{enumerate}
If consecutive hyperplanes in the chain are strongly separated, then
alternative~\textup{(2)} occurs.
\end{lemma}

\begin{proof}
The equality $\hk=\hh_i$ can occur for at most one index, so discard that
index if necessary. Suppose that $\hk$ and $\hh_i$ are not transverse. Since
$\kk\cap\h_i\neq\varnothing$, the possible nesting relations are
$\kk\subseteq\h_i$, $\kk^*\subseteq\h_i$, and $\h_i\subseteq\kk$.
The first relation can hold for only finitely many $i$. Indeed, if it held
for infinitely many indices, then every vertex of $\kk$ would belong to every
$\h_j$: for fixed $j$, choose $i\geq j$ with $\kk\subseteq\h_i\subseteq\h_j$.
The principal ultrafilter of such a vertex would contain the infinite
strictly descending chain $(\h_j)$, contrary to the descending chain
condition. The same argument, using a vertex of $\kk^*$, excludes the second
relation for infinitely many $i$.
It follows that, for all sufficiently large $i$, either $\hk$ is transverse
to $\hh_i$ or $\h_i\subseteq\kk$. If $\h_N\subseteq\kk$ for one sufficiently
large $N$, then $\h_i\subseteq\h_N\subseteq\kk$ for every $i\geq N$, giving
alternative~(2). If this never occurs, alternative~(1) holds. Finally, when
consecutive hyperplanes are strongly separated, alternative~(1) is impossible
because $\hk$ would be transverse to two consecutive hyperplanes.
\end{proof}

\begin{lemma}\label{lem:singleton-chain}
Let $\h_1\supsetneq\h_2\supsetneq\cdots$ be a descending chain whose
consecutive hyperplanes are strongly separated. Then
$\bigcap_{i\geq1}\mathcal U_{\h_i}=\{\xi\}$ for a unique point
$\xi\in\partial X$.
\end{lemma}

\begin{proof}
The sets $\mathcal U_{\h_i}$ are nonempty compact sets and form a descending
family, so their intersection is nonempty. A principal ultrafilter cannot
belong to the intersection, because it would contain the infinite strictly
descending chain $(\h_i)$. Thus, the intersection lies in the Roller boundary.

Suppose that $\alpha$ and $\beta$ both belong to the intersection and
$\alpha\neq\beta$. Choose a halfspace $\m\in\alpha$ with $\m^*\in\beta$.
Since $\m,\h_i\in\alpha$, the two halfspaces intersect for every $i$. Apply
Lemma~\ref{lem:nesting} to $\m$. If $\h_i\subseteq\m$ eventually, then
$\h_i\in\beta$ and upward consistency would force $\m\in\beta$, contradicting
$\m^*\in\beta$. Therefore, $\hm$ is transverse to every sufficiently large
$\hh_i$. It is then transverse to two consecutive hyperplanes in the chain,
contradicting strong separation. Hence $\alpha=\beta$.
\end{proof}

\begin{definition}[Regular point]\label{def:regular-point}
Let $X$ be irreducible. A point $\xi\in\partial X$ is \emph{regular} if
there is a descending chain $(\h_i)$ whose consecutive hyperplanes are
strongly separated and whose Roller clopen sets have intersection
$\{\xi\}$. Such a chain is called a defining chain of $\xi$. We denote the
set of regular points by $\bdreg X$.
\end{definition}

An isomorphism $\phi\colon X\to X'$ of irreducible CAT(0) cube complexes
carries $\HH(X)$ bijectively onto $\HH(X')$, preserving inclusion,
complementation and transversality, hence strong separation; it also
satisfies $\phi\,\mathcal U_{\h}=\mathcal U_{\phi\h}$.  Therefore $\phi$
carries defining chains to defining chains and $\bdreg X$ onto
$\bdreg X'$.  In particular, if $g$ is a cubical automorphism of
$X=X_1\times\cdots\times X_p$ carrying $X_k$ onto $X_r$, then
$g(\bdreg X_k)=\bdreg X_r$.

By Corollary~\ref{cor:pairwise-strong}, Definition~\ref{def:regular-point}
agrees with the chain characterization of regular points in
\cite[Proposition~5.10]{FernosLecureuxMatheus2018}, which is stated there for
an irreducible complex, and which we accordingly apply one factor at a time.

\begin{lemma}
\label{lem:absorption}
Let $\xi\in\bdreg X$ have defining chain $(\h_i)$.
\begin{enumerate}[label=\textup{(\alph*)}]
\item If $\kk\in\xi$, then $\h_i\subseteq\kk$ for every sufficiently large
$i$.
\item If $\eta\in\bdreg X$ is distinct from $\xi$ and $(\kk_j)$ is a defining
chain of $\eta$, then $\h_i\cap\kk_j=\varnothing$ for all sufficiently large
$i$ and $j$.
\end{enumerate}
\end{lemma}

\begin{proof}
For part (a), the halfspaces $\kk$ and $\h_i$ belong to the same ultrafilter,
so they intersect. If $\hk=\hh_i$ for some $i$, then $\kk=\h_i$, because an
ultrafilter cannot contain both orientations of one hyperplane; hence
$\h_j\subseteq\kk$ for every $j\geq i$. Otherwise Lemma~\ref{lem:nesting}
applies. The transverse alternative contradicts the strong separation of a
consecutive pair, so $\h_i\subseteq\kk$ eventually.

For part (b), choose a halfspace $\mathfrak b\in\xi$ with
$\mathfrak b^*\in\eta$. Part (a) gives $\h_i\subseteq\mathfrak b$ for all
sufficiently large $i$, and $\kk_j\subseteq\mathfrak b^*$ for all
sufficiently large $j$. Since halfspaces are being regarded as complementary
vertex sets, $\mathfrak b\cap\mathfrak b^*=\varnothing$. Hence
$\h_i\cap\kk_j=\varnothing$ for all sufficiently large $i,j$.
\end{proof}

\subsection{Regular elements and their poles}\label{subsec:regular-elements}

\begin{definition}\label{def:regular-element}
Let $X=X_1\times\cdots\times X_p$ and let a group $G$ act on $X$ by cubical
automorphisms. An element $g\in G$ is \emph{regular with skewering data} if
its induced cubical automorphism preserves every factor and, for each $k$,
there are halfspaces $\mathfrak b_k'\subseteq\mathfrak b_k$ of $X_k$ such
that $\widehat{\mathfrak b_k'}$ and $\widehat{\mathfrak b_k}$ are strongly
separated and $g\mathfrak b_k\subseteq\mathfrak b_k'$.  We keep the chosen
halfspaces as part of the data.
\end{definition}

\begin{lemma}\label{lem:regular-poles}
Let $g$ be regular with skewering data
$(\mathfrak b_k',\mathfrak b_k)$.  For each factor $X_k$, the families
$\mathfrak c_{k,r}=g^r\mathfrak b_k$ for $r\in\mathbb Z$ and
$\mathfrak d_{k,r}=g^{-r}\mathfrak b_k^*$ for $r\geq0$ are descending in
the indicated directions, and consecutive hyperplanes are strongly
separated. Consequently, there are unique regular points
$\xi_k^+(g),\xi_k^-(g)\in\bdreg X_k$ such that
\begin{equation*}
\{\xi_k^+(g)\}=\bigcap_{r\geq0}\mathcal U_{g^r\mathfrak b_k}
\quad\text{and}\quad
\{\xi_k^-(g)\}=\bigcap_{r\geq0}\mathcal U_{g^{-r}\mathfrak b_k^*}.
\end{equation*}
The two points are distinct. We call them the Roller poles of $g$ on $X_k$.
\end{lemma}

\begin{proof}
The skewering relation gives
$g\mathfrak b_k\subseteq\mathfrak b_k'\subseteq\mathfrak b_k$, so
$g^r\mathfrak b_k$ is descending as $r$ increases. Taking complements and
applying $g^{-1}$ gives $g^{-1}\mathfrak b_k^*\subseteq\mathfrak b_k^*$, so that the second family is descending.

We show that $\widehat{\mathfrak b_k}$ and $g\widehat{\mathfrak b_k}$ are
strongly separated. Since
$g\mathfrak b_k\subseteq\mathfrak b_k'\subsetneq\mathfrak b_k$, the halfspaces
$g\mathfrak b_k$ and $\mathfrak b_k$ are properly nested. Their bounding
hyperplanes are therefore distinct and non-transverse, hence disjoint.
If a hyperplane $\hm$ were transverse to both, then the inclusions
$g\mathfrak b_k\subseteq\mathfrak b_k'\subseteq\mathfrak b_k$ and
Lemma~\ref{lem:intermediate} would imply that $\hm$ is transverse to
$\widehat{\mathfrak b_k'}$. This contradicts the strong separation of
$\widehat{\mathfrak b_k'}$ and $\widehat{\mathfrak b_k}$. Applying powers of
$g$ proves strong separation of every consecutive pair in both chains.
Lemma~\ref{lem:singleton-chain} now gives the two asserted points.

The point $\xi_k^+(g)$ contains $\mathfrak b_k$, while $\xi_k^-(g)$ contains
$\mathfrak b_k^*$. An ultrafilter cannot contain both, so the poles are
distinct.
\end{proof}

The next proposition extracts from the random-walk argument of
Fern\'os--L\'ecureux--Math\'eus exactly the skewering data used below. We
include the extraction because the statement that regular elements exist does
not, by itself, record one common choice of strongly separated halfspaces in
all factors.

\begin{proposition}[Consequence of Fern\'os--L\'ecureux--Math\'eus]
\label{prop:FLM-regular}
Let $X$ be a non-Euclidean, finite-dimensional, second countable CAT(0) cube
complex, and let a countable group $G$ act essentially and non-elementarily
on $X$ by cubical automorphisms. Let $G^0$ be the subgroup preserving every
irreducible factor. Then $G^0$ contains an element regular with skewering
data in the sense of Definition~\ref{def:regular-element}.
\end{proposition}

\begin{proof}
Write $X=X_1\times\cdots\times X_p$ and fix vertices $o_k\in X_k^0$. Recall
that $X$ is finite-dimensional and second countable, so it satisfies the
standing convention of \cite{FernosLecureuxMatheus2018}, and $\HH(X)$, hence
every $\HH(X_k)$, is countable.
By \cite[Lemma~4.6]{FernosLecureuxMatheus2018}, the induced action of $G^0$
on every $X_k$ is essential and non-elementary. The
Fern\'os--L\'ecureux--Math\'eus results used below require only
admissibility.  We may nevertheless choose the measure with a finite first
moment.  Enumerate $G^0=\{g_1,g_2,\ldots\}$, put $o=(o_1,\ldots,o_p)$, and
assign to $g_j$ a positive weight proportional to $2^{-j}/(1+d(o,g_jo))$.
The resulting probability measure $\mu$ is admissible and has finite first
moment. Let $Z_n=s_1\cdots s_n$ be the associated right random walk. For each
factor, denote by $\check{\nu}_k$ the stationary measure for the reflected
measure $\check\mu$ on the Roller compactification of $X_k$.

We apply \cite[Theorem~7.1, Corollary~7.3 and
Theorem~8.1]{FernosLecureuxMatheus2018} to each factor separately.  Their
hypotheses hold: the action of $G^0$ on $X_k$ is essential and
non-elementary by the previous paragraph, and $X_k$ is irreducible, hence is
its own unique irreducible factor and is preserved by $G^0$, which is the
extra hypothesis of \cite[Theorem~8.1]{FernosLecureuxMatheus2018}.  We obtain
a set of sample paths of full measure such that, on every factor, $Z_n o_k$
converges to a regular point $\eta_k\in\bdreg X_k$. Since the action of $G^0$
on $X_k$ is non-elementary and $\check\mu$ is admissible whenever $\mu$ is,
the unique $\check\mu$-stationary measure $\check{\nu}_k$ is non-atomic; this
is recorded in the opening lines of the proof of
\cite[Theorem~8.1]{FernosLecureuxMatheus2018}.  Note also that the quantifier
in \cite[Lemma~11.4]{FernosLecureuxMatheus2018} is over \emph{every}
halfspace containing the limit point, so that the lemma applies to the
sample-path-dependent halfspaces chosen below. Fix one sample path having all
these properties.

For each $k$, choose a defining chain
$\mathfrak s_{k,1}\supsetneq\mathfrak s_{k,2}\supsetneq\cdots$ of
$\eta_k$.  The clopen sets $\mathcal U_{\mathfrak s_{k,m}}$ decrease to
$\{\eta_k\}$ by Lemma~\ref{lem:singleton-chain}.  Since $\check{\nu}_k$
is non-atomic, continuity from above gives
$\check{\nu}_k(\mathcal U_{\mathfrak s_{k,m}})\to0$.  The two clopen sets
$\mathcal U_{\mathfrak s_{k,m}}$ and
$\mathcal U_{\mathfrak s_{k,m}^*}$ partition the Roller compactification.
Choose $m(k)$ so that, for $\mathfrak s_k=\mathfrak s_{k,m(k)}$, one has
$\check{\nu}_k(\mathcal U_{\mathfrak s_k^*})>1-1/(2p)$.

Apply \cite[Lemma~11.3]{FernosLecureuxMatheus2018} to the chosen sample
path, its limit $\eta_k$, and the halfspace $\mathfrak s_k\in\eta_k$.  It
gives pairwise strongly separated halfspaces
$\mathfrak s_{k,2}'\subsetneq\mathfrak s_{k,1}'\subsetneq\mathfrak s_k$
and an integer $N_k$ such that, for $n>N_k$, either
$Z_n\mathfrak s_k\subseteq\mathfrak s_{k,2}'$ or
$Z_n\mathfrak s_k\supseteq(\mathfrak s_{k,2}')^*$.  Following that paper,
call $n$ an $\mathfrak s_k$-skewering time when the first inclusion holds,
and let $A_k$ be the set of all such times.  Lemma~11.4 of the same paper
gives
\begin{equation*}
 \liminf_{n\to\infty}\frac{1}{n}\#(A_k\cap\{1,\ldots,n\})
 \geq\check{\nu}_k(\mathcal U_{\mathfrak s_k^*})
 >1-\frac{1}{2p}.
\end{equation*}
For every $n$, the elementary union bound gives
\begin{equation*}
 \#\left(\bigcap_{k=1}^p A_k\cap\{1,\ldots,n\}\right)
 \geq\sum_{k=1}^p\#(A_k\cap\{1,\ldots,n\})-(p-1)n.
\end{equation*}
After division by $n$ and passage to lower limits, the preceding density
estimate yields
\begin{equation*}
 \liminf_{n\to\infty}\frac{1}{n}
 \#\left(\bigcap_{k=1}^p A_k\cap\{1,\ldots,n\}\right)>\frac12.
\end{equation*}
Choose $n>\max_kN_k$ belonging to every $A_k$, put $g=Z_n$, and set
$\mathfrak b_k=\mathfrak s_k$ and
$\mathfrak b_k'=\mathfrak s_{k,2}'$.  Then
$g\mathfrak b_k\subseteq\mathfrak b_k'$, and the hyperplanes bounding
$\mathfrak b_k$ and $\mathfrak b_k'$ are strongly separated.  Thus $g$ is
regular with skewering data in every factor.
\end{proof}

We need one further consequence of the skewering data.
Fern\'os--L\'ecureux--Math\'eus associate a unique visual point to every
squeezing Roller point, and the association is equivariant
\cite[Lemmas~6.8 and 6.10]{FernosLecureuxMatheus2018}.

\begin{lemma}\label{lem:poles-visual}
Let $g$ be regular with skewering data. Each pole $\xi_k^\pm(g)$ is a
contracting, hence squeezing, Roller point. Therefore there is a uniquely
associated visual point $\theta_k^\pm(g)\in\bdry X_k$, and the assignment
$\xi\mapsto\theta(\xi)$ is equivariant under cubical isometries.
\end{lemma}

\begin{proof}
Fix $k$ and put $\mathfrak c_r=g^r\mathfrak b_k$ for $r\in\mathbb Z$. The
chain is bi-infinite, descending, and pairwise strongly separated by
Lemma~\ref{lem:regular-poles} and Corollary~\ref{cor:pairwise-strong},
applied to the tail beginning at any given index.
Consecutive pairs are translates of one fixed pair, so their bridge lengths
are equal. The point $\xi_k^+(g)$ contains $\mathfrak c_r$ for every $r\geq0$
by definition. It also contains $\mathfrak c_r$ for $r<0$, because
$\mathfrak c_0\subseteq\mathfrak c_r$ and ultrafilters are upward closed.
Thus the bi-infinite chain witnesses that $\xi_k^+(g)$ is a contracting
Roller point in the sense of \cite[Remark~6.7]{FernosLecureuxMatheus2018}.
The same argument, applied to the reversed complementary chain, treats
$\xi_k^-(g)$. Contracting points are squeezing, and
\cite[Lemmas~6.8 and 6.10]{FernosLecureuxMatheus2018} give the unique visual
point and equivariance.
\end{proof}

\begin{lemma}\label{lem:pole-infinite-orbit}
Assume that $G\curvearrowright X$ is non-elementary, and let $g\in G^0$ be
regular with skewering data. Then every $G^0$-orbit of a pole
$\xi_k^\pm(g)$ is infinite.
\end{lemma}

\begin{proof}
Suppose that $G^0\xi_k^\pm(g)$ is finite. By
Lemma~\ref{lem:poles-visual} and equivariance, the orbit
$G^0\theta_k^\pm(g)$ in $\bdry X_k$ is finite. Fix basepoints
$o_i\in X_i$. If $\rho$ is a geodesic ray in $X_k$, then
$t\mapsto(o_1,\ldots,o_{k-1},\rho(t),o_{k+1},\ldots,o_p)$ is a geodesic
ray in $X$. Two such rays built from $\rho$ and $\rho'$ stay at distance
$d(\rho(t),\rho'(t))$, so the construction is well defined and injective on
asymptotic classes. Changing the fixed basepoints changes this ray by a
bounded distance, so this construction defines a canonical embedding
$\bdry X_k\hookrightarrow\bdry X$. The embedding is $G^0$-equivariant because
$G^0$ preserves every factor and therefore acts on $X$ as a product
$(h_1,\ldots,h_p)$ of automorphisms of the factors. Hence the embedded
$G^0$-orbit of $\theta_k^\pm(g)$ is finite.

The $G$-orbit of $\theta_k^\pm(g)$ is a union of $[G:G^0]<\infty$ many
$G^0$-orbits, hence finite.  It is therefore a finite $G$-orbit in
$\bdry X$, contradicting non-elementarity.
\end{proof}

\begin{theorem}\label{thm:general-position}
Let $X=X_1\times\cdots\times X_p$ be a non-Euclidean,
finite-dimensional, second countable CAT(0) cube complex, and let a countable
group $G$ act essentially and non-elementarily on $X$ by cubical
automorphisms. For each $k$,
let $T_k\subseteq\bdreg X_k$ be finite. Then there is an element $a\in G^0$,
regular with skewering data, such that
$\xi_k^+(a),\xi_k^-(a)\notin T_k$ for every $1\leq k\leq p$.
\end{theorem}

\begin{proof}
By Proposition~\ref{prop:FLM-regular}, choose one regular element $g\in G^0$
with poles $\xi_k^\pm(g)$. For each triple $(k,\epsilon,t)$ with
$\epsilon\in\{1,-1\}$ and $t\in T_k$, consider the bad set
$B(k,\epsilon,t)=\{h\in G^0:h\xi_k^\epsilon(g)=t\}$. If this set is
nonempty, choose $h_0\in B(k,\epsilon,t)$. Then
$B(k,\epsilon,t)=h_0\Stab_{G^0}(\xi_k^\epsilon(g))$. Indeed,
$h\xi_k^\epsilon(g)=h_0\xi_k^\epsilon(g)$ if and only if $h_0^{-1}h$
stabilizes the pole. By Lemma~\ref{lem:pole-infinite-orbit}, every stabilizer
appearing here has infinite index in $G^0$.

The bad sets are indexed by the triples $(k,\epsilon,t)$ with
$1\leq k\leq p$, $\epsilon=\pm1$ and $t\in T_k$.  Since $X$ is
finite-dimensional, it has finitely many irreducible factors, and each $T_k$
is finite by hypothesis, so there are at most
$2\sum_{k=1}^p\lvert T_k\rvert$ nonempty bad sets, each a coset of a
subgroup of infinite index; the cardinality of $\HH(X)$ plays no role.
Using~\cite{Neumann1954}, their
union cannot cover $G^0$. Choose $h\in G^0$ outside this union and put
$a=hgh^{-1}$. The element $a$ is regular with the conjugated skewering data.
Its poles are $h\xi_k^\pm(g)$, and the choice of $h$ gives the required
avoidance.
\end{proof}

\subsection{Infinite girth for cubulated groups}

We now prove the cubical theorem. The proof uses the action on the set
$Y=\bigsqcup_{k=1}^p X_k^0$, not on a boundary. The canonical factor
decomposition gives an action of $G$ on $Y$ in the sense that an element
which sends $X_k$ to $X_{k'}$ sends the corresponding vertex set to
$X_{k'}^0$.
We first record the elementary attraction property of the two tails of a
skewered chain.

\begin{lemma}\label{lem:band-attraction}
Let $g\in G^0$ be regular with skewering halfspaces $\mathfrak b_k$, and
put $\mathfrak c_{k,r}=g^r\mathfrak b_k$ for $r\in\mathbb Z$. For
$N\geq1$, let
$U_g(N)=\bigcup_{k=1}^p(\mathfrak c_{k,N}\cup\mathfrak c_{k,-N}^*)$.
Then $g^t(Y\setminus U_g(N))\subseteq U_g(N)$ whenever $|t|\geq2N$.
\end{lemma}

\begin{proof}
Fix $k$ and let
$v\in X_k^0\setminus(\mathfrak c_{k,N}\cup\mathfrak c_{k,-N}^*)$.
Since the chain $\mathfrak c_{k,r}$ is descending as $r$ increases, one has
$v\in\mathfrak c_{k,-N}\setminus\mathfrak c_{k,N}$. If $t\geq2N$, then
$g^tv\in\mathfrak c_{k,t-N}\subseteq\mathfrak c_{k,N}\subseteq U_g(N)$.
If $t\leq-2N$, the relation $v\notin\mathfrak c_{k,N}$ gives
$g^tv\notin\mathfrak c_{k,N+t}$. Since $N+t\leq-N$, one has
$\mathfrak c_{k,-N}\subseteq\mathfrak c_{k,N+t}$. Therefore
$g^tv\in\mathfrak c_{k,-N}^*\subseteq U_g(N)$. Taking the union over all
factors proves the lemma.
\end{proof}

\begin{theorem}\label{thm:cubical-girth}
Let $X$ be a non-Euclidean, finite-dimensional, second countable CAT(0) cube
complex. Let $G$ be a finitely generated group acting on $X$ essentially and
non-elementarily by cubical automorphisms. Then $\Gir(G)=\infty$.
\end{theorem}

\begin{proof}
Write $X=X_1\times\cdots\times X_p$ and let $G^0$ be the
factor-preserving subgroup. Fix an arbitrary finite generating set
$S=\{\gamma_1,\ldots,\gamma_\ell\}$ of $G$, with the identity omitted. We
shall verify Nakamura's criterion for this generating set.

\smallskip
\noindent\emph{Step 1.}
Apply Theorem~\ref{thm:general-position} with every $T_k$ empty. We obtain
a regular element $\sigma\in G^0$ with skewering halfspaces
$\mathfrak b_k$. Let $\mathfrak c_{k,r}=\sigma^r\mathfrak b_k$ for
$r\in\mathbb Z$, and denote its poles by $\xi_{k}^+(\sigma)$ and
$\xi_k^-(\sigma)$.

\smallskip
\noindent\emph{Step 2.}
Regard the sets of regular points of the factors as the disjoint union
$\bigsqcup_k\bdreg X_k$.
Consider all poles of $\sigma$ and all their
translates by the generators and their inverses, i.e.,
\begin{equation*}
P=\{\xi_k^\delta(\sigma):1\leq k\leq p,\ \delta=\pm1\}
\cup
\{\gamma_j^\epsilon\xi_k^\delta(\sigma):1\leq j\leq\ell,\
\ \epsilon,\delta=\pm1,\ 1\leq k\leq p\}.
\end{equation*}
The set $P$ is finite, with $\lvert P\rvert\leq2p(2\ell+1)$, and every point
of $P$ is regular by the invariance recorded after
Definition~\ref{def:regular-point}: a generator carrying $X_k$ onto $X_r$
restricts to an isomorphism of cube complexes and therefore carries
$\bdreg X_k$ onto $\bdreg X_r$.
For each $r$, let $T_r=P\cap\bdreg X_r$.
Theorem~\ref{thm:general-position} gives a regular element $\tau\in G^0$,
with skewering halfspaces $\mathfrak b'_r$, whose two poles on $X_r$ avoid
$T_r$. Put $\mathfrak c'_{r,s}=\tau^s\mathfrak b'_r$ for
$s\in\mathbb Z$. Thus every pole of $\tau$ is distinct from every point in
$P$, which lies in the same factor.

\smallskip
\noindent\emph{Step 3.}
For $N\geq0$, the halfspaces $\tau^N\mathfrak b'_r$ and
$\tau^{-N}\mathfrak b_r'^*$ are the $N$th terms of defining chains of
$\xi_r^+(\tau)$ and $\xi_r^-(\tau)$.  Similarly,
$\sigma^N\mathfrak b_k$ and $\sigma^{-N}\mathfrak b_k^*$ lie in defining
chains of the poles of $\sigma$, while their translates by
$\gamma_j^\epsilon$ lie in the defining chains of the corresponding points of
$P$.

Whenever one of the two $\tau$-chains and one of these $\sigma$-chains lie
in the same factor, their endpoints are distinct by the choice of $\tau$.
Lemma~\ref{lem:absorption}(b) therefore makes their sufficiently deep terms
disjoint.  Chains in different factors are disjoint as subsets of $Y$.
There are only finitely many pairs of chains, so there is $N_1$ such that,
for every $N\geq N_1$, both $\tau$-halfspaces are disjoint from all the
corresponding $\sigma$-halfspaces and their generator-translates.

Fix a vertex $x\in X_1^0$.  Every chain just mentioned, and every translate
of a $\tau$-chain by some $\gamma_j^\epsilon$, has no vertex in its total
intersection.  Since there are only finitely many such chains, there is
$N_2$ such that, for every $N\geq N_2$, the point $x$ lies in none of their
$N$th terms.

Choose $N\geq\max\{N_1,N_2,1\}$ and define
\begin{align*}
U_\sigma&=\bigcup_{k=1}^p
 \left(\sigma^N\mathfrak b_k\cup\sigma^{-N}\mathfrak b_k^*\right),\text{and }
U_\tau=\bigcup_{k=1}^p
 \left(\tau^N\mathfrak b'_k\cup\tau^{-N}\mathfrak b_k'^*\right).
\end{align*}
The choice of $N_1$ gives
$U_\tau\cap U_\sigma=\varnothing$
and 
$U_\tau\cap\gamma_j^\epsilon U_\sigma=\varnothing$ for all
$1\leq j\leq\ell$, and $\ \epsilon=\pm1$.
Replacing $\epsilon$ by $-\epsilon$ and applying $\gamma_j^\epsilon$ to the
second relation gives
$U_\sigma\cap\gamma_j^\epsilon U_\tau=\varnothing$ for every $j$ and
$\epsilon=\pm1$.  The choice of $N_2$ gives
\begin{equation*}
x\notin(U_\sigma\cup U_\tau)\cup
\bigcup_{\epsilon=\pm1}\bigcup_{j=1}^{\ell}
\gamma_j^\epsilon(U_\sigma\cup U_\tau).
\end{equation*}

\smallskip
\noindent\emph{Step 4}
Set $\widehat\sigma=\sigma^{2N}$ and
$\widehat\tau=\tau^{2N}$.  Lemma~\ref{lem:band-attraction} gives, for every
$t\neq0$,
\begin{equation*}
\widehat\sigma^t(Y\setminus U_\sigma)\subseteq U_\sigma
\quad\text{and}\quad
\widehat\tau^t(Y\setminus U_\tau)\subseteq U_\tau.
\end{equation*}

\smallskip
\noindent\emph{Step 5.}
The basepoint condition above is exactly \eqref{eq:nak-1}.  The established
disjointness relations imply
\begin{equation*}
\{x\}\cup U_\tau\cup
\bigcup_{\epsilon=\pm1}\bigcup_{j=1}^{\ell}\gamma_j^\epsilon U_\tau
\subseteq Y\setminus U_\sigma.
\end{equation*}
Applying the attraction relation for $\widehat\sigma$ proves
\eqref{eq:nak-2}.  Similarly,
\begin{equation*}
\{x\}\cup U_\sigma\cup
\bigcup_{\epsilon=\pm1}\bigcup_{j=1}^{\ell}\gamma_j^\epsilon U_\sigma
\subseteq Y\setminus U_\tau,
\end{equation*}
and the attraction relation for $\widehat\tau$ proves \eqref{eq:nak-3}.
Thus Theorem~\ref{thm:nakamura} applies to the two elements
$\widehat\sigma=\sigma^{2N}$ and $\widehat\tau=\tau^{2N}$, with the sets
$U_\sigma,U_\tau$ and the basepoint $x$.  Therefore, $G$ is noncyclic and
$\Gir(G)=\infty$.
\end{proof}
\section{Finite-index subgroups and virtual quotients}\label{sec:finite-index}

The preceding cubical argument suggests the more general question of whether
infinite girth passes from a finite-index subgroup to the ambient group.
Of course, the
cyclic degeneracy case has to be excluded since $\mathbb Z$ has index two in
the infinite
dihedral group $D_\infty$, while $\Gir(\mathbb Z)=\infty$ and
$\Gir(D_\infty)=2$.  Indeed, the translations form a proper subgroup of
$D_\infty$, so every finite generating set of $D_\infty$ contains a
reflection $s$, and $ss$ is a nonempty cyclically reduced word of length two
representing the identity.

In the opposite direction, the implication is well-known. If $K\leq G$ has
finite index $d$ and $\Gir(K)<\infty$, then $\Gir(G)\leq(2d-1)\Gir(K)$
\cite[Lemma~3.1]{Schleimer2000}; when $K$ is normal in $G$ the bound improves
to $\bigl(2\operatorname{Diam}(G/K)+1\bigr)\Gir(K)$
\cite[Corollary~3.2]{Schleimer2000}.  Equivalently, infinite girth passes to
finite-index subgroups.  This yields the following reduction.

\begin{proposition}\label{prop:finite-index-core}
Let $H\leq\Gamma$ be a finite-index subgroup.  Suppose that $H$ is
noncyclic and $\Gir(H)=\infty$.  Then
\[
K=\operatorname{Core}_\Gamma(H)
  =\bigcap_{\gamma\in\Gamma}\gamma H\gamma^{-1}
\]
is a noncyclic normal finite-index subgroup of $\Gamma$ and
$\Gir(K)=\infty$.
\end{proposition}

\begin{proof}
The subgroup $K$ is normal and of finite index in $\Gamma$, and it has finite
index in $H$.  Using \cite[Lemma~3.1]{Schleimer2000}, applied inside $H$, we
see that
$\Gir(K)=\infty$. It remains to show that $K$ is noncyclic.  If $K$ were
cyclic, then $H$ would be virtually cyclic.  Choose a finite-index cyclic
subgroup $C_0\leq H$ and set $C=\operatorname{Core}_H(C_0)$.  Then
$C\trianglelefteq H$ is cyclic and has finite index; write $r=[H:C]$.
Since $|H/C|=r$, we have $x^r\in C$
for every $x\in H$, and therefore $H$ satisfies the law $[x^r,y^r]=e$.
Since $[x^r,y^r]$ is a nontrivial word in the free group $F(x,y)$, the
noncyclic group $H$ satisfies a nontrivial law.  Therefore
\cite[Theorem~4.1]{Schleimer2000} gives $\Gir(H)<\infty$, contradicting
the hypothesis.
\end{proof}

The next construction allows a quotient of a finite-index subgroup to be
promoted to a quotient of the ambient group.
We give the details in order to fix the permutation convention.

\begin{proposition}\label{prop:induced-quotient}
Let $H\leq\Gamma$ have finite index $d$, and let
$\phi\colon H\twoheadrightarrow L$ be an epimorphism.  Then there is a
homomorphism $\Psi\colon\Gamma\longrightarrow L^d\rtimes\operatorname{Sym}(d)$
such that $\Psi(H)$ admits an epimorphism onto $L$.
\end{proposition}

\begin{proof}
Choose representatives $t_1=e,t_2,\ldots,t_d$ for the left cosets
$\Gamma/H$.  For $g\in\Gamma$, let $\sigma_g\in\operatorname{Sym}(d)$ be
defined by $g t_iH=t_{\sigma_g(i)}H$.  Let $g,h\in\Gamma$. From
$(gh)t_iH=g\bigl(ht_iH\bigr)=gt_{\sigma_{h}(i)}H=t_{\sigma_{g}\sigma_{h}(i)}H$
we obtain that
$\sigma_{gh}=\sigma_g\sigma_h$, and
hence
$\sigma_{gh}^{-1}=\sigma_h^{-1}\sigma_g^{-1}$.
For $1\leq i\leq d$, let
$c_i(g)=t_i^{-1}g\,t_{\sigma_g^{-1}(i)}$. Observe that
\[c_i(g)H=t_i^{-1}gt_{\sigma_{g}^{-1}(i)}H=t_i^{-1}t_{\sigma_g(\sigma_{g}^{-1}(i))}H=H.\]
Therefore, $c_i(g)\in H$ for all $1\le i\le d$. Moreover, for $g,h\in\Gamma$,
we have
\[
c_i(g)\,c_{\sigma_g^{-1}(i)}(h)
=\bigl(t_i^{-1}g\,t_{\sigma_g^{-1}(i)}\bigr)
 \bigl(t_{\sigma_g^{-1}(i)}^{-1}h\,t_{\sigma_h^{-1}\sigma_g^{-1}(i)}\bigr)
=t_i^{-1}(gh)\,t_{\sigma_{gh}^{-1}(i)}
=c_i(gh).
\]
Let $\operatorname{Sym}(d)$ act on $L^d$ by
$(\sigma\cdot(a_i))_i=a_{\sigma^{-1}(i)}$, so that in the corresponding
semidirect product the $i$-th coordinate of
$(\mathbf a,\sigma)(\mathbf b,\tau)$ equals $a_i\,b_{\sigma^{-1}(i)}$, and
the permutation coordinate equals $\sigma\tau$.
Now, for $g\in \Gamma$, define
$\Psi(g)=\big((\phi(c_1(g)),\ldots,\phi(c_d(g))),\sigma_g\big)$. It is easy
to see that $\Psi$ is a homomorphism.

If $h\in H$, then $ht_1H=hH=t_1H$, so $\sigma_h(1)=1$ and
$c_1(h)=t_1^{-1}h\,t_1=h$.  Thus $\Psi(H)$ is contained in the subgroup
\[
\operatorname{Stab}(1)=\{(\mathbf a,\sigma):\sigma(1)=1\}
\leq L^d\rtimes\operatorname{Sym}(d).
\]
If
$\sigma(1)=1$, then $\sigma^{-1}(1)=1$ and the first coordinate of
$(\mathbf a,\sigma)(\mathbf b,\tau)$ is $a_1b_1$.  Restricting the first
coordinate projection to
$\Psi(H)$ gives a homomorphism $\Psi(H)\to L$ with image
$L$.
\end{proof}

\begin{theorem}\label{thm:virtual-linear-quotient}
Let $\Gamma$ be finitely generated and let $H\leq\Gamma$ have finite index.
If $H$ admits an epimorphism onto a non-virtually-solvable linear group $L$,
then $\Gir(\Gamma)=\infty$.
\end{theorem}

\begin{proof}
Apply Proposition~\ref{prop:induced-quotient} and let
$Q=\Psi(\Gamma)$.  Since $\Gamma$ is finitely generated, so is $Q$.  Also
$[Q:\Psi(H)]\leq[\Gamma:H]$, because $\Psi$ is onto $Q$ and the cosets of
$\Psi(H)$ in $Q$ are the images of the cosets of $H$ in $\Gamma$.  The group
$\Psi(H)$ surjects onto $L$, so it is not virtually solvable; a group with a
non-virtually-solvable subgroup of finite index is not virtually solvable,
so $Q$ is not virtually solvable.  In particular, $Q$ is noncyclic.

Suppose $L\leq\operatorname{GL}_n(\mathbb K)$.  The group
$L^d\rtimes\operatorname{Sym}(d)$ is linear. Indeed, embed $L^d$ as block
diagonal matrices in $\operatorname{GL}_{nd}(\mathbb K)$ and realize
$\operatorname{Sym}(d)$ as permutations of the $d$ blocks. Thus $Q$ is a
finitely generated noncyclic non-virtually-solvable linear group, and
\cite[Theorem~4.4]{Akhmedov2005} gives $\Gir(Q)=\infty$.  Since $Q$ is a
noncyclic quotient of $\Gamma$, Proposition~\ref{prop:quotient} gives
$\Gir(\Gamma)=\infty$.
\end{proof}

Recall that a group is \emph{large} if it has a finite-index subgroup which
surjects onto a nonabelian free group.  Akhmedov and Mishra asked whether
every finitely generated large group has infinite girth
\cite[Question~1]{AkhmedovMishra2026}.  We answer this question in the
affirmative.

\begin{corollary}\label{cor:large-girth}
Every finitely generated large group has infinite girth.
\end{corollary}

\begin{proof}
Let $H\leq\Gamma$ be a finite-index subgroup admitting an epimorphism onto a
nonabelian free group $F$.  Since $\Gamma$ is finitely generated and $H$ has
finite index, $H$ and hence $F$ are finitely generated, so $F$ has finite
rank at least two.  The group $F$ embeds in
$\operatorname{GL}_2(\mathbb Z)$ and is not virtually solvable, so
Theorem~\ref{thm:virtual-linear-quotient} applies.
\end{proof}

\begin{remark}\label{rem:cubical-route}
Corollary~\ref{cor:large-girth} can also be deduced from
Theorem~\ref{thm:cubical-girth} in place of \cite{Akhmedov2005}.  Indeed,
$F^d\rtimes\operatorname{Sym}(d)$ acts properly on a product of $d$ copies of
the Cayley tree of $F$, permuting the factors, and this product is an
essential, non-Euclidean, finite-dimensional, locally finite CAT(0) cube
complex.  Since $\Gamma$ acts transitively on $\Gamma/H$, all $d$ coordinate
projections of $\Psi(\Gamma)\cap F^d$ are conjugate in the wreath product,
and one of them contains a finite-index subgroup of $F$; hence, all of them
are non-elementary, and the action of $\Psi(\Gamma)$ on the product is
essential and non-elementary.
\end{remark}

\begin{remark}\label{rem:finite-index-question}
The argument above does not settle the general finite-index problem.  If
$H\leq\Gamma$ is noncyclic, $[\Gamma:H]<\infty$, and $\Gir(H)=\infty$, then
Proposition~\ref{prop:finite-index-core} reduces the problem to the case in
which $H$ is normal.  Theorem~\ref{thm:virtual-linear-quotient} shows,
moreover, that if a counterexample does exist, the subgroup $H$ in this case
admits no epimorphism
onto a non-virtually-solvable linear group. Combined with
\cite[Lemma~3.1]{Schleimer2000}, the same applies to every finite-index
subgroup of $H$. In particular, no finite-index subgroup of $H$ maps onto a
non-virtually-solvable linear group.  Therefore, $H$ is not large.
\end{remark}
\begin{remark}
The above result was obtained independently by Akhmedov and Mahmudov~\cite{AM26} by similar arguments. We thank them for letting us know.     
\end{remark}
In connection with \cite[Question~2]{AkhmedovMishra2026}, we record the
following more general obstruction coming from extensions of groups
satisfying laws. We thank Azer Akhmedov for the following general version.
\begin{proposition}\label{prop:ZsemidirectGamma}
Let
$1\longrightarrow N\longrightarrow G\longrightarrow H\longrightarrow1$
be a short exact sequence, where $G$ is finitely generated.  Suppose that
$N$ and $H$ each satisfy a nontrivial law.  If
$G\not\cong\mathbb Z$, then $\Gir(G)<\infty$.
\end{proposition}
\begin{proof}
Choose a nontrivial law
$u=u(x_1,\ldots,x_r)\in F(x_1,\ldots,x_r)$
for $H$ and a nontrivial law
$v=v(y_1,\ldots,y_s)\in F(y_1,\ldots,y_s)$
for $N$.  For pairwise distinct variables $x_{i,j}$, let
$u_i=u(x_{i,1},\ldots,x_{i,r})$,
where $1\leq i\leq s$. 
Let
$W=v(u_1,\ldots,u_s)$. For every substitution of the variables $x_{i,j}$ by elements of $G$, the
image of each $u_i$ in $H$ is trivial.  Hence every $u_i$ evaluates to an
element of $N$, and the law $v$ shows that $W$ evaluates to the identity.
Thus $W$ is a law for $G$. It remains to verify that $W$ is nontrivial.  The free group on the
variables $x_{i,j}$ is the free product of the free groups on the $s$
disjoint blocks
$\{x_{i,1},\ldots,x_{i,r}\}$.
The elements $u_1,\ldots,u_s$ are nontrivial elements belonging to
distinct free factors.  Therefore the subgroup they generate is the free group
$\langle u_1\rangle*\cdots*\langle u_s\rangle$.  The substitution
$y_i\mapsto u_i$ is injective, so the nontriviality of $v$ implies that
$W\neq e$.

Thus $G$ satisfies a nontrivial law.  If $G$ is noncyclic, then
\cite[Theorem~4.1]{Schleimer2000} gives $\Gir(G)<\infty$.  If $G$ is
cyclic, the assumption $G\not\cong\mathbb Z$ implies that $G$ is finite,
and hence again has finite girth.
\end{proof}
We end this section by comparing the two mechanisms. Recall that the first mechanism is finite retraction discrimination, which for
countable groups is equivalent to mixed-identity-freeness by
\cite[Proposition~5.3]{HullOsin2016}.  The second uses the
finite ping-pong property $P_{\mathrm{FPP}}$ and Nakamura's criterion. The cubical argument verifies Nakamura's criterion
directly and does not require either finite retraction discrimination or
$P_{\mathrm{FPP}}$.

\begin{proposition}\label{prop:commuting-obstruction}
Let $G$ contain nontrivial commuting subgroups $A$ and $B$ with $B$ normal in
$G$.  Then, for any $a\in A\setminus\{e\}$ and $b\in B\setminus\{e\}$, the
element $[a,zbz^{-1}]$ is a nontrivial mixed identity for $G$.  In particular
$G$ is not mixed-identity-free, and $G$ does not have finite retraction
discrimination.
\end{proposition}

\begin{proof}
Choose $a\in A\setminus\{e\}$ and $b\in B\setminus\{e\}$.  The
free-product normal form shows that $w=[a,zbz^{-1}]$ is nontrivial in
$G*\langle z\rangle$.  Let $g\in G$.  Since $B$ is normal,
$gbg^{-1}\in B$, and since $A$ and $B$ commute, one has
$\varphi_g(w)=[a,gbg^{-1}]=e$.
Thus $w$ is a nontrivial mixed identity, so $G$ is not mixed-identity-free;
and since no retraction detects $w$, Definition~\ref{def:discrimination}
fails as well.  When $A$ is also normal, this is the direct-product case of
\cite[Proposition~5.4(b)]{HullOsin2016}.
\end{proof}

\begin{example}\label{ex:product-free-groups}
The group $F_2\times F_2$ contains the two factors as nontrivial commuting
normal subgroups, so Proposition~\ref{prop:commuting-obstruction} rules out
finite retraction discrimination.  Nevertheless, it acts properly and
essentially on the product $T_1\times T_2$ of the two Cayley trees.  We
verify that this action is non-elementary.

A finite orbit in $T_1\times T_2$ would be bounded and would therefore have a
unique circumcenter fixed by $F_2\times F_2$.  This is impossible because
each factor contains elements acting as nontrivial translations on its Cayley
tree.
Suppose now that there is a finite orbit in the visual boundary.  The kernel
of the permutation action on this orbit is a finite-index subgroup
$H\leq F_2\times F_2$ fixing a point
$\theta\in\partial_\infty(T_1\times T_2)$.  A geodesic ray representing
$\theta$ has the form
$t\mapsto(\rho_1(a_1t),\rho_2(a_2t))$, where $a_1,a_2\geq0$,
$a_1^2+a_2^2=1$, and a term with $a_i=0$ is constant.  At least one
coefficient is positive; assume $a_1>0$.  Then $\theta$ determines an
endpoint $\theta_1\in\partial_\infty T_1$, and the first-coordinate action of
$H$ fixes $\theta_1$.  In particular,
$H_1=H\cap(F_2\times\{e\})$ is a finite-index subgroup of the first copy of
$F_2$ and fixes $\theta_1$.

For completeness, the stabilizer in $F_2$ of an endpoint of its Cayley tree
is cyclic.  Fix an endpoint $\eta$ and an integer-valued Busemann function
$b_\eta$ on the vertices of the tree.  Every element $u$ stabilizing $\eta$
satisfies $b_\eta(ux)-b_\eta(x)=\beta_\eta(u)$ for an integer independent
of $x$, and
$\beta_\eta\colon\Stab_{F_2}(\eta)\to\mathbb Z$ is a homomorphism.  If
$\beta_\eta(u)=0$, then $u$ is not hyperbolic: a hyperbolic element fixing
$\eta$ translates along an axis having $\eta$ as one endpoint, and the
absolute value of its Busemann character equals its positive translation
length.  An automorphism of a tree is either elliptic or hyperbolic, so $u$
is elliptic.  The left action of $F_2$
on its Cayley tree is free on vertices and without inversions; hence $u=e$.
Thus $\beta_\eta$ is injective, and the stabilizer is isomorphic to a
subgroup of $\mathbb Z$, hence cyclic.

It follows that $H_1$ is cyclic.  This is impossible: by the
Nielsen--Schreier rank formula, a subgroup of index $d<\infty$ in $F_2$ is
free of rank $1+d\geq2$.  Therefore, the action on $T_1\times T_2$ is
non-elementary.  Theorem~\ref{thm:cubical-girth} gives
$\Gir(F_2\times F_2)=\infty$.
\end{example}
\section{Acknowledgements}
We thank Azer Akhmedov for taking the time to go over a near complete draft of this manuscript and for his various comments and suggestions. We also thank him for showing us Theorem~\ref{thm:akhmedov-one-projection} and Proposition~\ref{prop:ZsemidirectGamma}.

\bibliography{name}
\bibliographystyle{amsalpha}
\end{document}